%% file: main.tex
\documentclass[11pt]{amsart}
\input{preamble}

\input{supp-labels}

\begin{document}

\title[Facets of the doubly stochastic X-ray polytope]{Facets of the X-ray polytope of doubly stochastic matrices}

\author{Prashant Chaudhary}
\address{Independent Researcher, San Francisco Bay Area, CA, USA}
\email{prashant@prashantchaudhary.com}
\urladdr{https://orcid.org/0009-0001-1980-109X}

\author{Kapil Khandelwal}
\address{Department of Civil and Environmental Engineering and Earth Sciences, University of Notre Dame, Notre Dame, IN, USA}
\urladdr{https://orcid.org/0000-0002-5748-6019}

\date{October 2026}

\begin{abstract}
\input{b0_abstract}
\end{abstract}


\maketitle

\input{b1_intro}
\input{b2_prelim}
\input{b3_criterion}
\input{b4_constructions}
\input{b5_steps}
\input{b6_census}
\input{b7_conclusion}

\appendix
\input{bA_computation}

\section*{Acknowledgments}
Part of the computation was carried out on the computing resources of the Center for Research Computing at the University of Notre Dame.

\emph{Use of AI.} During the preparation of this work the authors used Claude (Anthropic) Opus and ChatGPT (OpenAI) Sol in order to check proofs, write and verify the computation code, search the literature, and help with drafting the text. After using these tools, the authors reviewed and edited the content as needed. The authors assume responsibility for all content.

\bibliographystyle{amsplain-names}
\bibliography{refs}

\end{document}

%% file: preamble.tex
\usepackage[margin=1.1in]{geometry}
\usepackage{amsmath,amssymb,amsthm,mathtools}
\usepackage{enumitem}
\usepackage{booktabs}
\usepackage[hidelinks]{hyperref}
\usepackage[capitalize,nameinlink]{cleveref}

\newtheorem{theorem}{Theorem}[section]
\usepackage{aliascnt}
\makeatletter
\newcommand\sharedthm[2]{\newaliascnt{#1}{theorem}\newtheorem{#1}[#1]{#2}\aliascntresetthe{#1}\@addtoreset{#1}{section}}
\makeatother
\sharedthm{proposition}{Proposition}
\sharedthm{lemma}{Lemma}
\sharedthm{corollary}{Corollary}
\sharedthm{conjecture}{Conjecture}
\theoremstyle{definition}
\sharedthm{definition}{Definition}
\sharedthm{example}{Example}
\sharedthm{problem}{Problem}
\theoremstyle{remark}
\sharedthm{remark}{Remark}

\crefname{conjecture}{Conjecture}{Conjectures}
\crefname{problem}{Problem}{Problems}

\makeatletter\let\@writetocindents\relax\makeatother

\input{macros}

%% file: macros.tex
\newcommand{\R}{\mathbb{R}}
\newcommand{\Z}{\mathbb{Z}}

\newcommand{\cD}{\mathcal{D}}

\newcommand{\cE}{\mathcal{E}}
\newcommand{\Ph}{\widehat{P}}
\newcommand{\DU}{\cD_{2}^{U}}

\newcommand{\law}{\operatorname{law}}

\newcommand{\spn}{\operatorname{span}}
\newcommand{\one}{\mathbf{1}}
\newcommand{\ind}[1]{\mathbf{1}\{#1\}}
\newcommand{\Es}{\cE^{*}}
\newif\ifanon

\newcommand{\VWEXACTK}{12}

%% file: b0_abstract.tex
We study the polytope of anti-diagonal X-rays of $k\times k$ doubly stochastic matrices, the image of the Birkhoff polytope under the map that sums the entries along each anti-diagonal. Its facets other than the coordinate facets are the minimal joint tail bounds for lattice laws of a sum of two dependent uniform random variables. We give an exact facet criterion in terms of the allowed edges of a tight graph. For rays of the underlying tropical fan, total support of the tight graph suffices for a facet but is not necessary, as an example with $k=19$ shows. Up to affine functions and positive scaling, the potentials that define a facet form a polytope whose vertices are its rays. Facets propagate by padding and by grid refinement under explicit hypotheses. We give infinite families of facets from no step potential, facets of quadratic height, superpolynomially many facets, and a classification of three-step potentials. Total support of a step potential is characterized by a Hall-type condition on block starts; with computer assistance, total support decides facets of step potentials for $k\le26$. The facet counts up to $k=12$ are found directly for $k\le3$ and, for $k\ge4$, come from an enumeration by a provably complete algorithm; an independent exact recomputation reproduces its ray lists for $5\le k\le12$.

%% file: b1_intro.tex
\section{Introduction}\label{sec:intro}

Add up the entries of a doubly stochastic matrix along each anti-diagonal. Which vectors can come out? For $3\times3$ matrices, with rows and columns numbered from $0$, one of the constraints fits on a line: every doubly stochastic $Z$ satisfies
\[
Z_{12}+Z_{21}+Z_{22}\le1+Z_{00}.
\]
Each entry on the left can be as large as one, yet their sum exceeds one only when the top-left entry is positive. This inequality defines a facet of the polytope of anti-diagonal sums. For $k\ge3$ the Birkhoff polytope itself has only $k^2$ facets, one for each entry; its image under this map already has 7,041 facets at $k=12$, and their number grows faster than any polynomial in $k$. This paper is about those facets.

Let $\Ph(k)$ be the set of \emph{anti-diagonal X-rays} of $k\times k$ doubly stochastic matrices: the vectors $x=(x(0),\dots,x(2k-2))$ with $x(t)=\sum_{i+j=t}Z_{ij}$ for a doubly stochastic matrix $Z$. It is the image of the Birkhoff polytope under a linear map, a polytope of dimension $2k-3$ whose vertices are X-rays of permutation matrices.

X-rays of permutations are classical objects of discrete tomography \cite{BebeacuaMansourPostnikovSeverini2005,BrualdiFritscher2014,Nordh2017}, and deciding whether a vector is the X-ray of a permutation is NP-complete \cite{BrunettiDelLungoGritzmannDeVries2008}. The doubly stochastic relaxation was studied recently by Li, Liu and Yao \cite{LiLiuYao2026}, in the equivalent convention of diagonal X-rays. For a \emph{potential} $u:\{0,\dots,k-1\}\to\R$ let $v(t)=\min_{i+j=t}(u(i)+u(j))$. Every permutation $\sigma$ has $v(i+\sigma(i))\le u(i)+u(\sigma(i))$, so the min-plus inequality
\[
I(u):\qquad\sum_{t=0}^{2k-2}v(t)\,x(t)\ \le\ 2\sum_{i=0}^{k-1}u(i)
\]
is valid on $\Ph(k)$. The $3\times3$ inequality above reads $x(3)+x(4)\le1+x(0)$; it is $I(u)$ for $u=(0,1,1)$, rewritten with $\sum_tx(t)=3$. Li, Liu and Yao show that these inequalities, with $x\ge0$ and the affine hull, describe $\Ph(k)$, that every facet is a coordinate inequality $x(t)\ge0$ or an $I(u)$ with $u$ on a ray of a tropical fan, and they compute the case $k=5$ \cite[Example~3.5]{LiLiuYao2026}. They note that ``it is unclear whether these rays have any simple descriptions.'' Not every ray gives a facet. We determine which do, and we describe large classes of them.

The inequality $I(u)$ has the flavor of the subadditivity conditions in the Gomory--Johnson theory of the infinite group relaxation and of master knapsack polyhedra \cite{AraozEvansGomoryJohnson2003,BasuHildebrandKoppe2016a,BasuHildebrandKoppe2016b,Gomory1969}: rays modulo affine functions play the part of extreme functions, and staircases the part of piecewise linear functions with few slopes. We know of no theorem there that applies directly, but the analogy suggests tools for the open problems of \cref{sec:conclusion}.

\subsection{Probabilistic motivation}
The polytope arises in the study of sums of dependent random variables. Let $X$ and $Y$ be uniform on $[0,1]$ with arbitrary dependence. A law on $\{1/k,\dots,(2k-1)/k\}$ is the law of such a sum $X+Y$ if and only if its mass vector, multiplied by $k$, lies in $\Ph(k)$ (\cref{prop:C}). Under this dictionary the inequality $I(u)$ for a step potential is a joint bound on a sum of tail probabilities of $X+Y$: a single step at $\eta$ gives $\mathbb{P}(X+Y\ge2\eta)+\mathbb{P}(X+Y\ge1+\eta)\le2(1-\eta)$, which is stronger than the sum of the two individual worst cases, and a potential with integer increments gives such a bound with repeated thresholds. In the companion paper \cite[Theorem~4.1]{CKcompanion} we show that a law on $[0,2]$ with mean one is the law of $X+Y$ if and only if it satisfies all of these joint tail bounds. On the grid $(1/k)\Z$ finitely many bounds suffice, and the facets of $\Ph(k)$ other than the coordinate facets are exactly the minimal list of joint tail bounds for lattice laws of $X+Y$. Convex order, which decides the analogous question for three or more uniform summands, accounts for only $2(k-2)$ of these facets: the convex hinges give $k-2$ of them and the concave hinges another $k-2$ (\cref{prop:hinges}). No result of this paper depends on the companion; the probabilistic reading is used only for motivation.

\subsection{Main results}
A potential $u$ that is not affine is a \emph{ray} if the potentials $z$ for which $z(i)+z(j)$ depends only on $i+j$ along the tight pairs of $u$ are exactly the combinations of $1$, $i$ and $u$ (\cref{def:ray}); these are the rays of the fan. After adding an affine function and scaling, a ray has nonnegative integer increments with minimum $0$ and greatest common divisor $1$; the largest increment is its \emph{height} (\cref{def:height}), and height one means a step potential. The pairs $(i,j)$ with $u(i)+u(j)=v(i+j)$ form the \emph{tight graph} of $u$; a tight edge is \emph{allowed} if it lies in a permutation all of whose pairs are tight, and the tight graph has \emph{total support} if every tight edge is allowed. A ray $u$ for which $I(u)$ defines a facet is a \emph{facet ray}. Our results on the facets are the following.
\begin{itemize}[leftmargin=*]
\item \emph{An exact facet criterion} (\cref{thm:empty,thm:F,thm:R}). If $u$ has a tight permutation, $I(u)$ defines a facet if and only if the allowed edges meet every anti-diagonal and a space of potentials attached to them, reduced by the bipartite components of the allowed graph, has dimension three. In particular one anti-diagonal without allowed edges excludes a facet, and a ray whose tight graph has total support gives a facet. Total support is not necessary: at $k=19$ a facet ray lacks it (\cref{ex:noTS}).
\item \emph{The fiber of a facet} (\cref{thm:fiber,rem:cubical}). Up to affine functions and positive scaling, the potentials that define a given facet form a polytope whose vertices are the rays that define it, and its dimension is the number of bipartite components of the allowed graph. So a facet is defined by a single ray iff that ray has total support, and at $k=19f$ some facet is defined by $2^f$ rays, the vertices of an $f$-dimensional cube.
\item \emph{Step potentials} (\cref{thm:S,thm:S4,thm:dec,thm:T3}). Single, double and consecutive steps give facets in both signs. Staircases whose block lengths after the first are nonincreasing give facets for every $k$; there are at least $\sum_{n<k}\mathrm{p}(n)-1$ of them, so the number of facets of $\Ph(k)$ grows faster than any polynomial (\cref{cor:dec}). Three-step potentials are classified: with block lengths $(g_0,\dots,g_3)$ and sign $+$, the staircase gives a facet unless $g_1,g_2,g_3$ are pairwise distinct with $g_1<g_3$. With a complement symmetry this classifies every step potential with at most three or at least $k-4$ steps.
\item \emph{Total support of step potentials} (\cref{thm:H,cor:probA,cor:probB}). Total support of a staircase fails iff some nonempty set $X$ of block starts other than the first has at most $|X|$ tight neighbors. With this, facet implies total support for staircases with at most seven or at least $k-8$ steps (computer-assisted, \cref{thm:Eflat6}), for unimodal and several other classes of tails of every length (\cref{prop:Eflatclasses}), and by exact computation for $k\le26$; total support implies a ray for $k\le30$. So total support decides the step potentials for $k\le26$, and the remaining statements concern the sequence of block ends alone.
\item \emph{Padding and refinement} (\cref{thm:pad,thm:refine,prop:descent}). Let $u$ be a ray whose tight graph is connected and has total support. If $u$ satisfies a condition $(\mathrm{P}_L)$ on its flat steps, padding it with a flat step at the left gives a facet of $\Ph(k+1)$, and if it satisfies the mirror condition $(\mathrm{P}_R)$, padding it with a steepest step at the right does; the padded potential again satisfies the hypotheses, so padding can be repeated (\cref{cor:pad}). Under the same hypotheses on connectivity and total support, refining the grid by a factor $m$ gives a facet of $\Ph(mk)$: a facet-defining step function on $(1/k)\Z$ that satisfies them stays facet-defining on every grid $(1/(mk))\Z$. The hypotheses hold for every facet ray with $k\le12$, and they are not automatic: a ray without $(\mathrm{P}_L)$ pads to a non-ray.
\item \emph{Beyond steps} (\cref{thm:fam,thm:height}). Explicit families of potentials that are not step potentials give facets for all large $k$, and the height of facets grows at least quadratically: some facet ray of $\Ph(k)$ has height at least $\lfloor(k-2)/3\rfloor^2$ for $k\ge8$.
\item \emph{Exact counts} (\cref{sec:census,app:computation}). The facet counts of $\Ph(k)$ for $k=2,\dots,12$ are $2,6,13,23,45,85,187,426,1033,2633,7041$. The count for $k=5$ is in \cite[Example~3.5]{LiLiuYao2026}; the values for $k\ge6$ are new. For $k\ge4$ they come from an exact enumeration of the fan (for $k\le3$ the facets are found directly): we walk the facet graph of a Minkowski sum of simplices whose normal fan is the fan, listing the ridges at each facet as flats of a difference matroid. Started from a certified facet, the walk provably visits every facet (\cref{prop:complete}). An independent walk on the vertices of the sum, in exact integer arithmetic, reproduces every ray list for $5\le k\le\VWEXACTK$, and a floating-point version of it agrees.
\end{itemize}

\subsection{Methods}
The face of $\Ph(k)$ exposed by $I(u)$ is the image under the X-ray map of a face of the Birkhoff polytope, namely the doubly stochastic matrices supported on the allowed graph. Faces of the Birkhoff polytope are described by Brualdi and Gibson \cite{BrualdiGibson1977} through bipartite graphs with total support (see also \cite{Paffenholz2015}); the facet criterion pushes this description forward through a linear map that aggregates the entries along anti-diagonals. So the questions become questions about bipartite matchings in the tight graph (Hall's theorem, alternating cycles, the Dulmage--Mendelsohn decomposition \cite{DulmageMendelsohn1958}) and about the linear relations that the tight pairs impose on potentials. For staircases the tight graph is determined by the sequence of block ends, and total support reduces to a Hall-type condition on block starts. On the enumerative side, the fan of Li, Liu and Yao is the normal fan of a Minkowski sum of $2k-1$ simplices \cite{DevelinSturmfels2004,GritzmannSturmfels1993}, and its rays are the facet normals of that sum; we enumerate them by an exact, certified walk on its facet graph.

At $k=12$, 5,748 of the 7,041 facets come from no step potential; padding explains 3,182 of them (\cref{tab:cores} and \cref{prob:C}). The convex hinges $u(i)=(i-j)^+$ give $k-2$ facets, the convex and concave hinges together $2(k-2)$ (\cref{prop:hinges}), and for X-rays with all entries at most one the convex hinges alone decide membership in $\Ph(k)$ (\cref{cor:smallatoms}).

\subsection{Notation}
\Cref{tab:notation} lists the main symbols. Throughout, $k$ is the size of the matrices; $m$ is always a refinement factor (\cref{sec:refinement}).

\begin{table}[ht]
\centering
\small
\caption{Main notation.}
\label{tab:notation}
\begin{tabular}{@{}ll@{}}
\toprule
Symbol & Meaning\\
\midrule
$\Ph(k)$ & X-rays of $k\times k$ doubly stochastic matrices (\cref{sec:count})\\
$x(t)$ & mass on the anti-diagonal $i+j=t$, $0\le t\le2k-2$\\
$u$, $v=u\,\square\,u$ & potential on $\{0,\dots,k-1\}$; $v(t)=\min_{i+j=t}(u(i)+u(j))$\\
$I(u)$ & the inequality $\sum_tv(t)x(t)\le2\sum_iu(i)$ (\cref{sec:duality})\\
$\one$, $\iota$ & constant potential $1$; potential $\iota(i)=i$\\
$\cE(u)$ & tight graph: pairs with $u(i)+u(j)=v(i+j)$\\
$\Es(u)$ & allowed graph: tight edges that lie in a tight permutation\\
total support & $\Es(u)=\cE(u)$ and a tight permutation exists\\
$Z(u)$, $Z^*(u)$ & $z$ with $z(i)+z(j)$ a function of $i+j$ on $\cE(u)$, resp.\ $\Es(u)$\\
ray & non-affine $u$ with $Z(u)=\spn\{\one,\iota,u\}$ (\cref{def:ray})\\
height & largest increment of a normalized integer ray (\cref{def:height})\\
$\operatorname{bip}(u)$ & number of bipartite components of $\Es(u)$ (\cref{sec:criterion})\\
$u_J$, $u(g)$ & staircase with step set $J$, with block lengths $g=(g_0,\dots,g_r)$\\
$S_l$, $E_l$ & first and last index of block $l$ of a staircase\\
$M(s)$ & largest $E_p+E_q$ with $p+q=s$ (\cref{sec:steps})\\
$\operatorname{tri}(p,q)$ & number of anti-diagonals of tight pairs in the corner of blocks $p$, $q$ (\cref{lem:corner})\\
thresholds & steps $j/k$ in the tail-sum reading (\cref{rem:multiplicity})\\
$(\mathrm{P}_L)$, $(\mathrm{P}_R)$ & padding conditions (\cref{sec:padding})\\
$m$, $\operatorname{spr}_m$, $\operatorname{ref}_m$ & refinement factor; spread and refinement by $m$ (\cref{sec:refinement})\\
$f$, $u_f$ & multiplicity of the cubical family; its potential on $k=19f$ points (\cref{prop:cubical})\\
$Q=\sum_tQ_t$ & Minkowski sum of the simplices $Q_t$; its facet normals are the rays (\cref{sec:census})\\
core & facet ray that is not an end padding of a facet ray (\cref{sec:padding})\\
\bottomrule
\end{tabular}
\end{table}

\subsection{Organization}
\Cref{sec:prelim} sets up the polytope, potentials, rays and height. \Cref{sec:criterion} proves the facet criterion and describes the fiber of a facet, \cref{sec:constructions} treats padding, refinement, the non-step families and height, \cref{sec:steps} treats step potentials, \cref{sec:census} reports the exact counts and the results on total support of staircases, and \cref{sec:conclusion} lists open problems. \Cref{app:computation} describes the certified computations. The supplement contains the longer proofs and case analyses, the remaining census tables and the certificate formats.

%% file: b2_prelim.tex
\section{The polytope and potentials}\label{sec:prelim}\label{sec:grid}

\subsection{Count coordinates}\label{sec:count}
Let $\Ph(k)$ be the set of \emph{anti-diagonal X-rays} $x=(x(0),\dots,x(2k-2))$ of $k\times k$ doubly stochastic matrices $Z$, where $x(t)=\sum_{i+j=t}Z_{ij}$. Equivalently, $x\in\Ph(k)$ if and only if the linear system in the unknowns $Z_{ij}\ge0$ consisting of the $2k$ line-sum equations $\sum_jZ_{ij}=1$, $\sum_iZ_{ij}=1$ and the $2k-1$ anti-diagonal equations $\sum_{i+j=t}Z_{ij}=x(t)$ is feasible; its constraint matrix stacks the row-and-column incidence matrix of the complete bipartite graph $K_{k,k}$ on top of the incidence matrix of the anti-diagonals, and only the right-hand side $(\one,\one,x)$ depends on $x$. So $\Ph(k)$ is the image of the Birkhoff polytope under the linear \emph{X-ray map} $A:Z\mapsto x$. Its affine hull is $\{\sum_t x(t)=k,\ \sum_t t\,x(t)=k(k-1)\}$ and $\dim\Ph(k)=2k-3$ for $k\ge2$ (\cref{cor:affhull}; $\Ph(1)$ is a point). A \emph{facet} of $\Ph(k)$ is a face of dimension $2k-4$, and the inequalities $x(t)\ge0$ are the \emph{coordinate inequalities}.

Write $\DU$ for the set of laws of $X+Y$ over all couplings of two random variables $X,Y$ uniform on $[0,1]$. The following proposition is the dictionary between $\Ph(k)$ and these laws.

\begin{proposition}[Lattice laws and X-rays]\label{prop:C}
A law $F$ on $\{1/k,2/k,\dots,(2k-1)/k\}$ belongs to $\DU$ if and only if the vector with entries $x(t)=k\,F(\{(t+1)/k\})$, $t=0,\dots,2k-2$, lies in $\Ph(k)$. No law in $\DU$ supported on $(1/k)\Z$ charges $0$ or $2$.
\end{proposition}

\begin{proof}
If $S=X+Y\in(1/k)\Z$ almost surely, put $I=\lfloor kX\rfloor$ and $J=\lfloor kY\rfloor$. The fractional parts of $kX$ and $kY$ then add to $1$ almost surely, so $kS=I+J+1$, and $(I,J)$ is a coupling of two uniform laws on $\{0,\dots,k-1\}$. Conversely, given such a coupling and an independent $U$ uniform on $[0,1]$, the variables $X=(I+U)/k$ and $Y=(J+1-U)/k$ are uniform on $[0,1]$ and $k(X+Y)=I+J+1$. Multiplying the coupling probabilities by $k$ gives a doubly stochastic matrix, whose X-ray is $x$. A mass at $0$ or $2$ would need an atom of a uniform marginal.
\end{proof}

By \cref{prop:C}, every linear inequality valid on $\Ph(k)$, divided by $k$, is a constraint on the lattice laws in $\DU$, and the facets of $\Ph(k)$ are the irredundant ones. \Cref{prop:C} is also \cite[Proposition~3.1]{CKcompanion} (proof included). Doubly stochastic matrices are the discrete copulas on the uniform grid \cite{PerroneSolusUhler2019}, which gives another reading of \cref{prop:C}.

\subsection{Potentials and the inequality \texorpdfstring{$I(u)$}{I(u)}}\label{sec:duality}
A \emph{potential} is a function $u:\{0,\dots,k-1\}\to\R$. Put
\[
v(t)=(u\,\square\,u)(t)=\min_{i+j=t}\bigl(u(i)+u(j)\bigr),\qquad t=0,\dots,2k-2,
\]
the discrete infimal convolution of $u$ with itself. Every permutation $\sigma$ satisfies $v(i+\sigma(i))\le u(i)+u(\sigma(i))$ termwise, so
\[
I(u):\qquad \sum_{t=0}^{2k-2}v(t)\,x(t)\ \le\ 2\sum_{i=0}^{k-1}u(i)
\]
is valid on $\Ph(k)$. Li, Liu and Yao \cite[Proposition~3.1]{LiLiuYao2026} showed that these inequalities, together with $x\ge0$ and the affine hull, describe $\Ph(k)$. This is finite linear programming duality for the assignment problem: a linear functional $c$ on X-rays pulls back to the assignment cost $c(i+j)$, an optimal dual pair $(\alpha,\beta)$ can be replaced by the potential $u=(\alpha+\beta)/2$ because the cost is symmetric, and then $c\le v$ (\cref{lem:griddual}). The potentials $\pm\one$ and $\pm\iota$, where $\iota(i)=i$, give the two equations of the affine hull.

In the reading of \cref{prop:C}, let $S$ have a law on $\{1/k,\dots,(2k-1)/k\}$, put $T=kS-1$, so that $x(t)=k\,\mathbb{P}(T=t)$, and let $U_k$ be uniform on $\{0,\dots,k-1\}$. Then $I(u)$ divided by $k$ reads $\mathbb{E}\,v(T)\le2\,\mathbb{E}\,u(U_k)$. Let $\tilde u$ be the upper semicontinuous step function on $[0,1]$ equal to $u(i)$ on $(i/k,(i+1)/k)$ and to $\max(u(i-1),u(i))$ at $i/k$, with $u(-1)=u(k)=-\infty$. Then $(\tilde u\,\square\,\tilde u)((t+1)/k)=v(t)$ for $0\le t\le2k-2$ and $\int_0^1\tilde u=\frac1k\sum_iu(i)$, so on lattice laws $I(u)$ divided by $k$ is the inequality $\mathbb{E}(\tilde u\,\square\,\tilde u)(X+Y)\le2\,\mathbb{E}\tilde u(X)$, which holds for every coupling of two uniforms \cite[Sections~3.2 and~3.3]{CKcompanion}. Li, Liu and Yao index X-rays by the diagonals $i-j=d$; reversing the columns, $j\mapsto k-1-j$, turns their diagonals into our anti-diagonals and their results into the statements used here.

\subsection{Tight graphs and rays}\label{sec:tight}\label{sec:potentials}
The \emph{tight graph} of $u$ is $\cE(u)=\{(i,j):u(i)+u(j)=v(i+j)\}$, a symmetric set of pairs; every anti-diagonal $\{i+j=t\}$ carries at least one of its edges, namely a minimizing pair. A permutation attains equality in $I(u)$ iff all pairs $(i,\sigma(i))$ lie in $\cE(u)$; call such permutations \emph{tight}. The \emph{allowed graph} $\Es(u)\subseteq\cE(u)$ is the set of edges that lie in at least one tight permutation; its edges are \emph{allowed}. We say $\cE(u)$ has \emph{total support} if a tight permutation exists and $\Es(u)=\cE(u)$. This is total support of the $0/1$ matrix of $\cE(u)$ in the sense of Sinkhorn and Knopp \cite{SinkhornKnopp1967}: every nonzero entry lies on a diagonal of nonzero entries; by their theorem it holds iff the matrix can be scaled to a doubly stochastic matrix with the same zero pattern. When a tight permutation exists, the allowed edges are the edges inside the irreducible blocks of the Dulmage--Mendelsohn decomposition of the bipartite graph of $\cE(u)$ \cite{DulmageMendelsohn1958}; we compute them from the switching digraph of \cref{lem:D}.

Let $Z(u)$ be the space of $z\in\R^k$ such that $z(i)+z(j)$ depends only on $i+j$ along $\cE(u)$, and define $Z^*(u)$ in the same way with $\Es(u)$ in place of $\cE(u)$. Always $\spn\{\one,\iota,u\}\subseteq Z(u)\subseteq Z^*(u)$.

\begin{definition}\label{def:ray}
A non-affine potential $u$ is a \emph{ray} if $Z(u)=\spn\{\one,\iota,u\}$.
\end{definition}

\begin{remark}[Rays of the tropical fan]\label{rem:fan}
Li, Liu and Yao \cite[Proposition~3.4]{LiLiuYao2026} fix $u(0)=u(k-1)=0$ (which removes the lineality space $\spn\{\one,\iota\}$) and consider the arrangement of the tropical hypersurfaces $\mathcal{T}_t$, $1\le t\le 2k-3$, where $\mathcal{T}_t$ is the set of potentials for which the minimum defining $v(t)$ is attained by at least two distinct unordered pairs $\{i,j\}$. Its cones are the closures of the sets of potentials with a fixed set of minimizing pairs on every anti-diagonal. The potentials with the same minimizing pairs as $u$ form a relatively open subset of $Z(u)$ containing $u$, so the cone through $u$ spans $Z(u)$; it is one-dimensional modulo the lineality space iff $\dim Z(u)=3$. So our rays are exactly the rays of their arrangement, and their Proposition~3.4 says: \emph{every facet of $\Ph(k)$ is a coordinate inequality $x(t)\ge0$ or $I(u)$ for a ray $u$.} Not every ray gives a facet. The rest of the paper uses this statement; its proof is recalled in \cref{sec:SM-fan}.
\end{remark}

For $J\subseteq\{1,\dots,k-1\}$ the \emph{step potential} is $u_J(i)=\#\{j\in J:j\le i\}$. The sign $+$ refers to $u_J$ and the sign $-$ to $-u_J$. A \emph{staircase} is a step potential with sign $+$; by \cref{lem:A} below, every signed step potential gives the same inequality as some staircase, so \cref{sec:steps} works with staircases only.

\begin{definition}[Height]\label{def:height}
The space $Z(u)$ is cut out by linear equations with integer coefficients, so for a ray $u$ the line $Z(u)/\spn\{\one,\iota\}$ is spanned by the class of an integer vector. Hence every ray is, after adding an affine function and multiplying by a positive constant, a potential $u$ with $u(0)=0$ whose increments $d(i)=u(i+1)-u(i)$, $0\le i\le k-2$, are nonnegative integers with minimum $0$ and greatest common divisor $1$. The largest increment $\max_id(i)$ is the \emph{height} of the ray.
\end{definition}

By \cref{lem:A} the inequality, the tight graph and $Z(u)$ do not change under the affine change, and a positive multiple defines the same inequality. If $u$ and $-u$ are both rays, they have the same height, because adding $H\iota$ to $-u$, with $H$ the height of $u$, gives the increments $H-d(i)$. Height one means a step potential. A ray $u$ for which $I(u)$ defines a facet of $\Ph(k)$ is a \emph{facet ray}; we usually take it in the form of \cref{def:height}.

\begin{remark}[Facets as tail-sum bounds]\label{rem:multiplicity}
For $u$ as in \cref{def:height}, $u=\sum_{j=1}^{k-1}d(j-1)\,u_{\{j\}}$ is a sum of single steps in which the step at $j$ is repeated $d(j-1)$ times. In the reading of \cref{prop:C}, $I(u)$ divided by $k$ is a joint bound on a sum of tail probabilities of $X+Y$ in which the threshold $j/k$ occurs with multiplicity $d(j-1)$ \cite[Lemma~4.1 and Remark~4.2]{CKcompanion}; for distinct thresholds this is \cref{cor:tails}. So the height of a facet ray is the largest multiplicity of a threshold in its tail-sum bound, and by \cref{rem:fan} the facets of $\Ph(k)$ other than the coordinate facets form the minimal list of such bounds on the grid $(1/k)\Z$. The facets of height at least two, 5,748 of the 7,018 facets of $\Ph(12)$ other than the coordinate facets (\cref{tab:shapes}), need repeated thresholds on that grid.
\end{remark}

\subsection{Lemmas}

\begin{lemma}[Affine invariance and symmetries]\label{lem:A}
If $u'(i)=u(i)+\alpha+\beta i$, then $v'(t)=v(t)+2\alpha+\beta t$, and $I(u')$ and $I(u)$ define the same face of $\Ph(k)$, with $\cE(u')=\cE(u)$. Because $u_J+u_{J^c}=\iota$, the sign $-$ at $J$ gives the same inequality as the sign $+$ at the complement $J^c$. The reflection $i\mapsto k-1-i$ maps $\Ph(k)$ to itself by $x(t)\mapsto x(2k-2-t)$ and sends $(J,+)$ to $(k-J,-)$, where $k-J=\{k-j:j\in J\}$.
\end{lemma}

\begin{proof}
The right-hand side changes by $2k\alpha+\beta k(k-1)$, which equals the change of the left-hand side on the affine hull $\{\sum_tx(t)=k,\ \sum_t t\,x(t)=k(k-1)\}$. The rest is direct.
\end{proof}

\begin{lemma}[Face dimension]\label{lem:T}
Let $\mathcal{F}(u)$ be the face of $\Ph(k)$ on which $I(u)$ is tight, and assume it is nonempty. Then $\dim \mathcal{F}(u)=(2k-1)-\dim W(u)$, where $W(u)$ is the space of $c\in\R^{2k-1}$ for which there exist $\alpha,\beta\in\R^k$ with $c(i+j)=\alpha(i)+\beta(j)$ for every $(i,j)\in\Es(u)$.
\end{lemma}

\begin{proof}
The face of the Birkhoff polytope on which the pulled-back functional is tight is the convex hull of the tight permutation matrices. It equals the set of doubly stochastic matrices supported on $\Es(u)$, because the vertices of the latter are permutation matrices supported in $\Es(u)\subseteq\cE(u)$, all tight. The average of tight permutations that together cover $\Es(u)$ is positive on all of $\Es(u)$, so the affine hull of this face is the set of matrices supported on $\Es(u)$ with unit line sums; its direction space is $\ker B$, where $B$ is the row-and-column incidence matrix of $\Es(u)$. The X-ray map $A$ sends the face onto $\mathcal{F}(u)$, so $\dim \mathcal{F}(u)=\dim A(\ker B)$. A vector $c$ annihilates $A(\ker B)$ iff $A^{\top}c$ lies in the row space of $B$, that is, iff $c(i+j)=\alpha(i)+\beta(j)$ on $\Es(u)$.
\end{proof}

\Cref{lem:T} is the Brualdi--Gibson description of faces of the Birkhoff polytope \cite{BrualdiGibson1977} (see \cite{Paffenholz2015} for the combinatorics of these faces), pushed forward by the X-ray map. The lemma needs the allowed graph $\Es(u)$; with the tight graph $\cE(u)$ in its place it fails (\cref{sec:SM-prelim-ex}).

\begin{lemma}[Total support from a universal vertex]\label{lem:TS}
Let $G$ be a symmetric bipartite graph on rows and columns $\{0,\dots,k-1\}$. Suppose $G$ has a perfect matching $\sigma$ with $\sigma(e)=e$ for an index $e$ adjacent to every column. Then every edge of $G$ lies in a perfect matching.
\end{lemma}

\begin{proof}
Form the digraph on $\{0,\dots,k-1\}$ with an arc $i\to \sigma^{-1}(j)$ for each edge $(i,j)$ of $G$. An edge $(i,j)$ with $j=\sigma(i')$ lies in a perfect matching iff $i=i'$ or the arc $i\to i'$ lies on a directed cycle (switch along the alternating cycle). From $e$ there is an arc to every vertex, and every vertex $i$ has an arc to $\sigma^{-1}(e)=e$, because $(i,e)\in G$ by symmetry. So the digraph is strongly connected.
\end{proof}

The same argument without the universal vertex gives:

\begin{lemma}[Switching digraph]\label{lem:D}
If $\sigma$ is a tight permutation and the digraph $D_\sigma$ with an arc $i\to\sigma^{-1}(j)$ for every tight edge $(i,j)$ is strongly connected, then $\cE(u)$ has total support.
\end{lemma}

\begin{lemma}[Reduction to a potential space]\label{lem:U}
The allowed graph $\Es(u)$ is symmetric, since the inverse of a tight permutation is tight. Suppose $\Es(u)$ contains $(e,e)$ for an index $e$ adjacent in $\Es(u)$ to every index, and meets every anti-diagonal. Then $W(u)$ and $Z^*(u)$ are linearly isomorphic. Consequently, if $u$ is not affine, $I(u)$ is a facet if and only if $Z^*(u)=\spn\{\one,\iota,u\}$.
\end{lemma}

\begin{proof}
Given $c\in W(u)$ with $\alpha,\beta$, symmetry gives also $c(i+j)=\alpha(j)+\beta(i)$ on $\Es(u)$, so $z=(\alpha+\beta)/2$ satisfies $c(i+j)=z(i)+z(j)$ on $\Es(u)$. This $z$ is determined by $c$: $z(e)=c(2e)/2$ and $z(i)=c(e+i)-z(e)$. Conversely every $z\in Z^*(u)$ defines $c(t)$ on every $t$, because every anti-diagonal carries an edge. The two maps are linear and mutually inverse. If $u$ is not affine, $\one,\iota,u$ are independent, so $\dim W(u)\ge3$ and $\dim \mathcal{F}(u)\le 2k-4$, with equality iff $Z^*(u)=\spn\{\one,\iota,u\}$.
\end{proof}

\begin{corollary}\label{cor:affhull}
For $k\ge2$, $\dim\Ph(k)=2k-3$, and its affine hull is $\{\sum_tx(t)=k,\ \sum_tt\,x(t)=k(k-1)\}$.
\end{corollary}

\begin{proof}
For constant $u$, $\Es(u)$ is the complete graph and $Z^*(u)$ is the space of affine sequences (compare $(i,j+1)$ with $(i+1,j)$). By \cref{lem:U} with any $e$, $\dim W(u)=2$, and \cref{lem:T} with $\mathcal{F}(u)=\Ph(k)$ gives the dimension; the two equations are valid for every permutation.
\end{proof}

\begin{lemma}[Coordinate facets]\label{lem:coord}
For $k\ge4$, each of the $2k-1$ inequalities $x(t)\ge0$ defines a facet of $\Ph(k)$, and these $2k-1$ facets are distinct.
\end{lemma}

\begin{proof}[Idea of the proof]
The face $\{x(t)=0\}$ is the X-ray image of the face of the Birkhoff polytope supported on the complete bipartite graph minus one anti-diagonal matching; every edge of this graph lies in a perfect matching, so the argument of \cref{lem:T} applies, and for $k\ge4$ the comparisons of $(i,j+1)$ with $(i+1,j)$ leave exactly three free parameters. Distinct $t$ give distinct facets, because two coordinate functionals are not positive multiples of each other modulo the affine hull. The proof is in \cref{sec:SM-coord}.
\end{proof}

\begin{lemma}[Block tools]\label{lem:K}
Let $[s,t]$ be a block of consecutive indices.
\begin{enumerate}[label=(B\arabic*)]
\item If $[s,t]\times[s,t]\subseteq\Es(u)$, every $z\in Z^*(u)$ is affine on $[s,t]$.
\item If $\Es(u)$ contains every $(i,j)\in[s,t]\times[s,t]$ with $i+j\ge s+t$, every $z\in Z^*(u)$ is affine on $[s,t]$.
\end{enumerate}
\end{lemma}

\begin{proof}
(B1) Compare $(i,j+1)$ with $(i+1,j)$. (B2) Compare $(i,t)$ with $(i+1,t-1)$; both have sum $i+t\ge s+t$.
\end{proof}

%% file: b3_criterion.tex
\section{The facet criterion and the fiber of a facet}\label{sec:criterion}
We view a symmetric set $G$ of pairs as a graph on $\{0,\dots,k-1\}$ with an edge $\{i,j\}$ for every $(i,j)\in G$, a loop at $i$ when $(i,i)\in G$. A connected component of this graph is \emph{bipartite} if it contains no closed walk of odd length; a loop is a closed walk of length one, and an isolated vertex is a bipartite component. Let $\operatorname{bip}(u)$ be the number of bipartite components of the allowed graph $\Es(u)$, and
\[
Y^*(u)=\bigl\{y\in\R^k:\ y(i)+y(j)=0\ \text{for every }(i,j)\in\Es(u)\bigr\}.
\]
Then $Y^*(u)\subseteq Z^*(u)$, and $\dim Y^*(u)=\operatorname{bip}(u)$: along an edge $y(j)=-y(i)$, so on a component $y$ is determined by its value at one vertex $i_0$, a walk of length $\ell$ from $i_0$ to $j$ gives $y(j)=(-1)^\ell y(i_0)$, and these values are consistent for an arbitrary $y(i_0)$ if the component is bipartite, while a closed walk of odd length through a vertex $i$ forces $y(i)=-y(i)=0$ otherwise.

\begin{theorem}[Empty anti-diagonals]\label{thm:empty}
Let $k\ge2$ and let $u$ be a potential.
\begin{enumerate}[label=(\alph*)]
\item No inequality $I(u)$ defines a facet of the form $\Ph(k)\cap\{x(t)=0\}$.
\item If some anti-diagonal $t$ carries no edge of $\Es(u)$, then $I(u)$ does not define a facet of $\Ph(k)$.
\end{enumerate}
\end{theorem}

\begin{proof}[Idea of the proof]
(a) If $I(u)$ defined the facet $\{x(t)=0\}$, then after an affine change $v=-\lambda\ind{s=t}$ with $\lambda>0$, and the inequalities $u(i)+u(j)\ge v(i+j)$ at pairs avoiding the anti-diagonal $t$ force $u\ge0$ and hence $v(t)\ge0$. (b) The face of $I(u)$ is spanned by X-rays of tight permutations, so it lies in the proper face $\{x(t)=0\}$, which is a facet by (a) only if the two coincide. The proof is in \cref{sec:SM-empty}.
\end{proof}

So, for every potential, a single anti-diagonal without allowed edges certifies that $I(u)$ is not a facet. The other condition of the criterion below concerns potentials that expose the same face.

\begin{theorem}[A symmetric dual potential]\label{thm:symdual}
Let $u$ have a tight permutation, and suppose that every anti-diagonal carries an edge of $\Es(u)$. Then there is a potential $w$ with $\sum_iw(i)=\sum_iu(i)$ and $w(i)+w(j)\ge v(i+j)$ for all $i,j$, with equality if and only if $(i,j)\in\Es(u)$. Every such $w$ has $\min_{i+j=t}\bigl(w(i)+w(j)\bigr)=v(t)$ for all $t$, so $I(w)$ is the inequality $I(u)$; moreover $\cE(w)=\Es(u)$, the tight graph of $w$ has total support, and $u-w\in Y^*(u)$.
\end{theorem}

\begin{proof}[Idea of the proof]
Perturb $u$ by $\varepsilon h$ in the rows and by $-\varepsilon h\circ\sigma^{-1}$ in the columns, where $h$ is a linear extension of the order of the strong components of the switching digraph $D_\sigma$ of a tight permutation $\sigma$; this keeps the allowed pairs tight and makes the other tight pairs slack, and symmetrizing gives $w$. The remaining claims follow because $\Es(u)$ meets every anti-diagonal and contains the tight permutations of $u$. The proof is in \cref{sec:SM-symdual}.
\end{proof}

The potential $w$ is a strictly complementary dual solution of the assignment problem with weights $v(i+j)$, symmetrized. \Cref{thm:symdual} says that a failure of total support, when it does not leave an anti-diagonal empty, is invisible in the inequality: $I(u)$ is also $I(w)$ for a potential $w$ with total support, and $u$ and $w$ differ by an element of $Y^*(u)$.

\begin{theorem}[Facet criterion]\label{thm:F}
Let $u$ be a potential with a tight permutation. Then $I(u)$ defines a facet of $\Ph(k)$ if and only if
\begin{enumerate}[label=(\roman*)]
\item every anti-diagonal carries an edge of $\Es(u)$, and
\item $\dim Z^*(u)-\operatorname{bip}(u)=3$.
\end{enumerate}
Under (i), $\dim\mathcal{F}(u)=(2k-1)-\bigl(\dim Z^*(u)-\operatorname{bip}(u)\bigr)$.
\end{theorem}

\begin{proof}
Condition (i) is necessary by \cref{thm:empty}(b); assume it. By \cref{lem:T}, $\dim\mathcal{F}(u)=(2k-1)-\dim W(u)$. For $z\in Z^*(u)$ let $\Lambda z\in\R^{2k-1}$ be given by $(\Lambda z)(i+j)=z(i)+z(j)$ for $(i,j)\in\Es(u)$; this is well defined by the definition of $Z^*(u)$ and by (i). Clearly $\Lambda z\in W(u)$, with $\alpha=\beta=z$. Conversely, if $c\in W(u)$ with witnesses $\alpha,\beta$, then, $\Es(u)$ being symmetric, also $c(i+j)=\alpha(j)+\beta(i)$ on $\Es(u)$, so $z=(\alpha+\beta)/2$ satisfies $c(i+j)=z(i)+z(j)$ on $\Es(u)$; hence $z\in Z^*(u)$ and, by (i), $c=\Lambda z$. So $W(u)=\Lambda(Z^*(u))$. The kernel of $\Lambda$ consists of the $z$ with $z(i)+z(j)=0$ on $\Es(u)$, which is $Y^*(u)$, of dimension $\operatorname{bip}(u)$. Hence $\dim W(u)=\dim Z^*(u)-\operatorname{bip}(u)$, which gives the formula, and $\mathcal{F}(u)$ is a facet iff this number is $3$.
\end{proof}

The number $\dim Z^*(u)-\operatorname{bip}(u)$ is always at least $2$, because $\Lambda\one=2\cdot\one$ and $(\Lambda\iota)(t)=t$ are independent; it equals $2$ exactly when $\mathcal{F}(u)=\Ph(k)$. When $\Es(u)$ contains a loop $(e,e)$ at an index $e$ that is adjacent in $\Es(u)$ to every index, the allowed graph is connected and not bipartite, so $\operatorname{bip}(u)=0$, and \cref{thm:F} reduces to \cref{lem:U}.

\begin{proposition}\label{prop:raybip}
Let $u$ be a ray with a tight permutation. Then the tight graph $\cE(u)$, viewed as a graph, has no bipartite component.
\end{proposition}

\begin{proof}
Let $y(i)+y(j)=0$ for every $(i,j)\in\cE(u)$; we show $y=0$, which is the claim by the dimension count at the beginning of \cref{sec:criterion} applied to $\cE(u)$. Such a $y$ lies in $Z(u)$, so $y=\alpha+\beta\iota+\gamma u$. Suppose $\gamma\ne0$. Every anti-diagonal carries a tight edge, and on it $0=y(i)+y(j)=2\alpha+\beta t+\gamma v(t)$, so $v(t)=a+bt$ is affine. Put $\ell(i)=a/2+bi$, so that $\ell(i)+\ell(j)=a+b(i+j)=v(i+j)$. Taking $i=j$ in $u(i)+u(j)\ge v(i+j)$ gives $2u(i)\ge v(2i)=2\ell(i)$, so $u\ge\ell$. For a tight permutation $\sigma$, $\sum_i\bigl(u(i)+u(\sigma(i))\bigr)=\sum_iv(i+\sigma(i))=\sum_i\bigl(\ell(i)+\ell(\sigma(i))\bigr)$, so $\sum_i(u-\ell)(i)=0$, hence $u=\ell$ is affine, which a ray is not. So $\gamma=0$, and $2\alpha+\beta t=0$ on every anti-diagonal $t$, so $\alpha=\beta=0$.
\end{proof}

\begin{theorem}[Ray plus total support]\label{thm:R}
Let $u$ be a ray whose tight graph has total support. Then $I(u)$ defines a facet of $\Ph(k)$.
\end{theorem}

\begin{proof}
Total support gives a tight permutation and $\Es(u)=\cE(u)$. Every anti-diagonal carries a tight edge, hence an allowed one, and $Z^*(u)=Z(u)=\spn\{\one,\iota,u\}$ has dimension $3$. By \cref{prop:raybip}, $\operatorname{bip}(u)=0$, and \cref{thm:F} applies.
\end{proof}

\Cref{thm:R} needs no universal vertex. When there is one, \cref{lem:TS} gives total support and \cref{lem:U} computes $Z(u)$; the proofs of \cref{thm:S} below are of this kind.

\begin{corollary}\label{cor:facetbip}
If $I(u)$ defines a facet and the tight graph of $u$ does not have total support, then $\operatorname{bip}(u)\ge1$ and $\dim Z^*(u)\ge4$.
\end{corollary}

\begin{proof}
A facet is nonempty, so $u$ has a tight permutation, and every anti-diagonal carries an allowed edge by \cref{thm:empty}(b). Take $w$ from \cref{thm:symdual}. If $\operatorname{bip}(u)=0$, then $Y^*(u)=0$, so $u=w$, whose tight graph has total support, contrary to the hypothesis. Then \cref{thm:F}(ii) gives $\dim Z^*(u)=3+\operatorname{bip}(u)\ge4$.
\end{proof}

By \cref{rem:fan} only rays matter. For $6\le k\le12$ a ray gives a facet exactly when its tight graph has total support: this was checked on all 113,842 rays (\cref{sec:census}). The converse of \cref{thm:R} is nevertheless false in general.

\begin{example}\label{ex:noTS}
Let $k=19$, $O=\{2,4,5,7,9,10,11,12,13\}$ and $\chi=\ind{i\in O}$, so that $\chi$ vanishes on $O^c=\{0,1,3,6,8,14,15,16,17,18\}$. The sums of two elements of $O^c$ cover $\{0,\dots,36\}$ except $5,10,13,27$, and these four are sums of an element of $O^c$ and an element of $O$ ($5=0+5$, $10=0+10$, $13=0+13$, $27=16+11$). Hence $v(t)=1$ for $t\in\{5,10,13,27\}$ and $v(t)=0$ otherwise. A tight pair $(11,j)$ needs $1+\chi(j)=v(11+j)\le1$, so $\chi(j)=0$ and $11+j\in\{13,27\}$; since $\chi(2)=1$, the only tight neighbor of $11$ is $16$. So every tight permutation exchanges $11$ and $16$, and the $19$ tight pairs that contain $16$ but not $11$ lie in no tight permutation: the tight graph does not have total support. By exact computation, $\chi$ is a ray ($\dim Z(\chi)=3$), every anti-diagonal carries an allowed edge, and the face of $I(\chi)$ has dimension $34=2k-4$, so $I(\chi)$ defines a facet of $\Ph(19)$. In the terms of \cref{thm:F}, the edge $\{11,16\}$ is a bipartite component of $\Es(\chi)$, $\operatorname{bip}(\chi)=1$ and $\dim Z^*(\chi)=4$. After adding $\iota$, $\chi$ has increments in $\{0,1,2\}$, so it has height two. \Cref{prop:cubical} below proves all of these statements, for $f=1$, without computation (proof in \cref{sec:SM-cubical}).
\end{example}

\begin{lemma}[Distinct rays, distinct facets]\label{lem:distinct}
Let $u$ be a ray whose tight graph has total support, and let $u'$ be a potential that is not affine. Then $I(u)$ and $I(u')$ define the same face of $\Ph(k)$ if and only if $u'=\alpha+\beta\iota+\gamma u$ with $\alpha,\beta\in\R$ and $\gamma>0$. Equivalently, with the increments $d(i)=u(i+1)-u(i)$, $0\le i\le k-2$, and the \emph{pattern}
\[
p(u)=\frac{d-\min d}{\max d-\min d}\in[0,1]^{k-1},
\]
they define the same face if and only if $p(u)=p(u')$.
\end{lemma}

\begin{proof}
If $u'=\alpha+\beta\iota+\gamma u$ with $\gamma>0$, then $I(u')$ is $\gamma$ times $I(u)$ up to an identity that holds on the affine hull (\cref{lem:A}), so the faces agree. Conversely, a permutation $\sigma$ is tight for $u$ if and only if its X-ray lies on the face of $I(u)$, because for this X-ray $\sum_tv(t)x(t)=\sum_iv(i+\sigma(i))\le\sum_i\bigl(u(i)+u(\sigma(i))\bigr)$, with equality iff every pair $(i,\sigma(i))$ is tight. So equal faces have the same tight permutations, and $\cE(u)=\Es(u)=\Es(u')\subseteq\cE(u')$. Every $z\in Z(u')$ satisfies the relations along $\cE(u')$, in particular those along $\cE(u)$, so $u'\in Z(u')\subseteq Z(u)=\spn\{\one,\iota,u\}$ and $u'=\alpha+\beta\iota+\gamma u$; $\gamma\ne0$ because $u'$ is not affine. If $\gamma<0$, then $\cE(u)\subseteq\cE(u')=\cE(-u)$, since adding affine functions and multiplying by a positive number do not change the tight graph. Every anti-diagonal carries a pair of $\cE(u)$, which then both minimizes and maximizes $u(i)+u(j)$ on that anti-diagonal; so $u(i)+u(j)$ is constant on every anti-diagonal, and comparing $(i,i+2)$ with $(i+1,i+1)$ shows that $u$ is affine, a contradiction. Finally, $u\mapsto\alpha+\beta\iota+\gamma u$ with $\gamma>0$ acts on increments by $d\mapsto\beta+\gamma d$, which leaves $p(u)$ unchanged, and two non-constant increment vectors with the same pattern are related by such a map.
\end{proof}

For facets, the following theorem describes all potentials that define a given facet, with or without total support.

\begin{theorem}[The fiber of a facet]\label{thm:fiber}
Let $I(u_0)$ define a facet $F$ of $\Ph(k)$, let $v_0$ be the infimal convolution of $u_0$, $b=\operatorname{bip}(u_0)$, and
\[
D(F)=\Bigl\{w\in\R^k:\ w(i)+w(j)\ge v_0(i+j)\ \text{for all }i,j,\ \ \textstyle\sum_iw(i)=\sum_iu_0(i)\Bigr\}.
\]
\begin{enumerate}[label=(\alph*)]
\item $D(F)$ is a polytope of dimension $b$ with affine hull $u_0+Y^*(u_0)$. Every $w\in D(F)$ has infimal convolution $v_0$, so $I(w)$ is the inequality $I(u_0)$, and $\Es(w)=\Es(u_0)$.
\item A point $w\in D(F)$ lies in the relative interior of $D(F)$ if and only if $\cE(w)$ has total support. The smallest face of $D(F)$ containing $w$ has dimension equal to the number of bipartite components of the graph $\cE(w)$, and $\dim Z(w)$ exceeds that number by $3$.
\item A potential $u$ has $I(u)$ defining $F$ if and only if $u=\alpha+\beta\iota+\gamma w$ with $\gamma>0$ and $w\in D(F)$; then $w$ is unique, and $u$ is a ray if and only if $w$ is a vertex of $D(F)$. So the rays that define $F$, up to affine functions and positive scaling, correspond bijectively to the vertices of $D(F)$.
\item If $b=0$, then $F$ is defined by exactly one ray, which has total support. If $b\ge1$, then $F$ is defined by at least $b+1$ distinct rays, none of which has total support. In particular, every facet ray of $\Ph(k)$ has total support if and only if distinct facet rays define distinct facets of $\Ph(k)$.
\end{enumerate}
\end{theorem}

\begin{proof}[Idea of the proof]
For $w\in D(F)$ the slacks along a tight permutation of $u_0$ are nonnegative and sum to zero, so $w$ is tight along $\Es(u_0)$ and $w-u_0\in Y^*(u_0)$; the potential of \cref{thm:symdual} is an interior point, which gives (a). The inequalities of $D(F)$ that are not tight on all of it are those for pairs outside $\Es(u_0)$, which gives the first part of (b), and the map $\Lambda$ of the proof of \cref{thm:F} computes $\dim Z(w)$. Part (c) uses that two inequalities defining one facet differ by a positive factor and an affine function, and (d) counts vertices of a $b$-dimensional polytope. The proof is in \cref{sec:SM-fiber}.
\end{proof}

\begin{remark}[Distinct rays, one facet]\label{rem:samefacet}
In \cref{ex:noTS} let $\chi'=\chi-\epsilon_{11}+\epsilon_{16}$, the indicator of $(O\setminus\{11\})\cup\{16\}$, where $\epsilon_a(i)=\ind{i=a}$. Here $b=1$, and $D(F)$ is the segment $\{\chi+\lambda(\epsilon_{11}-\epsilon_{16}):-1\le\lambda\le0\}$ (the loop inequalities at $11$ and $16$ are its only constraints). Its endpoints $\chi$ and $\chi'$ are the two rays that define the facet $F$ of $I(\chi)$, the only tight neighbor of $16$ for $\chi'$ is $11$, and the midpoint has total support. So the hypothesis of total support in \cref{lem:distinct} cannot be dropped. For $k\le12$ every facet ray has total support (\cref{sec:census}). So the smallest $k$ for which $\Ph(k)$ has a facet ray without total support lies between $13$ and $19$. A search over the rays of height at most two for $k\le18$, with ranks computed modulo two large primes, found no facet ray without total support. This search is evidence only: it covers only height at most two, and a rank computed modulo a prime can be smaller than the rational rank, which can make a ray look like a non-ray, so the search could miss a facet ray. The bounds $13$ and $19$ do not depend on it.
\end{remark}

The example extends to every number of bipartite components. In the following proposition $f\ge1$ is a multiplicity: each index of \cref{ex:noTS} is replaced by a block of $f$ consecutive indices.

\begin{proposition}[Cubical fibers]\label{prop:cubical}\label{rem:cubical}
Let $O$ and $\chi$ be as in \cref{ex:noTS}, let $f\ge1$, $k=19f$, and $u_f(i)=\chi(\lfloor i/f\rfloor)$ for $0\le i<k$. Write $B_q=\{fq,\dots,fq+f-1\}$ for $0\le q\le18$ (the \emph{blocks}), and put
\[
\begin{gathered}
S=\{0,1,3,6,8,14,15,17,18\},\qquad O'=O\setminus\{11\},\\
T=\{5,10,13,27\},\qquad T_f=\{f(\tau+1)-1:\tau\in T\},
\end{gathered}
\]
$V_0=\bigcup_{s\in S}B_s$, $V_1=\bigcup_{q\in O'}B_q$, $I_a=11f+a$, $J_a=17f-1-a$ and $y_a=\epsilon_{I_a}-\epsilon_{J_a}$ for $0\le a<f$; so $B_{11}=\{I_a\}$, $B_{16}=\{J_a\}$, and $u_f$ is $0$ on $V_0\cup B_{16}$ and $1$ on $V_1\cup B_{11}$. For $q\in O$ let $N_q=\{s\notin O:q+s\in T\}$. Let
\[
\begin{gathered}
A=(V_0\times V_0)\cup L\cup L^{\top}\cup\{(I_a,J_a),(J_a,I_a):0\le a<f\},\quad\text{where}\\
L=\{(fq+a,\,fs+b):q\in O',\ s\in N_q,\ a+b=f-1\}.
\end{gathered}
\]
\begin{enumerate}[label=(\alph*)]
\item $v=\ind{t\in T_f}$, and $\cE(u_f)$ consists of the pairs of indices where $u_f$ vanishes together with the pairs $(i,j)$ with $u_f(i)+u_f(j)=1$ and $i+j\in T_f$. The only tight neighbor of $I_a$ is $J_a$.
\item $\Es(u_f)=A$. The $19f^2$ ordered pairs of $\cE(u_f)\setminus A$ are the pairs of zeros of $u_f$ that meet $B_{16}$. Every anti-diagonal carries an edge of $A$; the bipartite components of $A$ are the $f$ edges $\{I_a,J_a\}$, so $\operatorname{bip}(u_f)=f$.
\item $Z^*(u_f)=\spn\{\one,\iota,u_f,y_0,\dots,y_{f-1}\}$ has dimension $f+3$, and $Z(u_f)=\spn\{\one,\iota,u_f\}$; so $u_f$ is a ray.
\item $I(u_f)$ defines a facet $F_f$ of $\Ph(k)$; its dimension is $2k-4$.
\item $D(F_f)=\bigl\{u_f+\sum_a\lambda_ay_a:\lambda\in[-1,0]^f\bigr\}$, an $f$-cube. Its $2^f$ vertices are $0/1$ potentials; up to affine functions and positive scaling they are exactly the rays $u$ for which $I(u)$ defines $F_f$, and none of them has total support.
\item $C=\{u:A\subseteq\cE(u)\}$ is a closed cone of the fan of \cref{rem:fan}. Modulo $\spn\{\one,\iota\}$ it is the cone over $D(F_f)\times\{1\}$: it has dimension $f+1$ and $2^f$ extreme rays, so it is not simplicial for $f\ge2$.
\end{enumerate}
\end{proposition}

\begin{proof}[Idea of the proof]
Everything follows from two facts about the levels: the sums of two elements of $S$ are exactly $\{0,\dots,36\}\setminus T$, and for $i\in B_p$, $j\in B_q$ one has $i+j\in T_f$ if and only if $p+q\in T$ and the offsets of $i$ and $j$ in their blocks add up to $f-1$. These give $v$ and the tight pairs in (a). In (b), the forced pairing of $I_a$ with $J_a$ excludes the other pairs that meet $B_{16}$, and a spanning tree of $S$ formed by one edge from each $N_q$ gives Hall's condition with surplus one for $L$, which shows that every pair of $A$ lies in a tight permutation. In (c), the relations along $A$ make $z$ affine on $V_0$ with one slope and force $z-\alpha-ci$ to be constant on $V_1$, leaving only the free directions $y_a$; the loops at $J_a$ remove these from $Z(u_f)$. Then (d) is \cref{thm:F}, (e) follows from \cref{thm:fiber} and the loop inequalities at $I_a$ and $J_a$, and (f) from (e). The full proof is in \cref{sec:SM-cubical}.
\end{proof}

All statements of \cref{prop:cubical} were also confirmed for $f\le6$ ($k\le114$) by an exact computation independent of the proof (\cref{sec:SM-cubical}).

\begin{remark}[Padding the cube]\label{rem:padcount}
Let $w$ be a vertex of $D(F_f)$ in \cref{prop:cubical} and $\hat w=w+\iota$, which is normalized in the sense of \cref{sec:padding} below, where $\operatorname{pad}_L$ and $\operatorname{pad}_R$ are defined, with increments in $\{0,1,2\}$ and $H=2$. Then $\operatorname{pad}_L^a\operatorname{pad}_R^b\hat w$ is a ray for all $a,b\ge0$, and distinct choices of the vertex and of $(a,b)$ give distinct increment words (\cref{sec:SM-padcount}). Hence for $k\ge19$ the fan has at least $\sum_{f=1}^{\lfloor k/19\rfloor}2^f(k-19f+1)$ rays of height two, counted up to affine functions and positive scaling. This bound counts rays of the fan, not facets.
\end{remark}

%% file: b4_constructions.tex
\section{Padding, refinement, families and height}\label{sec:constructions}

\subsection{Padding}\label{sec:padding}
Under the hypotheses of \cref{thm:pad} below, a facet of $\Ph(k)$ given by a potential propagates to $\Ph(k+1)$ by adding a flat step at the left end or a steepest step at the right end of the potential. By \cref{lem:A} we may add an affine function to a potential without changing its inequality, its tight graph or $Z(u)$. So for a non-affine potential $u$ on $\{0,\dots,k-1\}$ we may assume that its increments $d(i)=u(i+1)-u(i)$ have minimum $0$; we call such $u$ \emph{normalized}, and we put $H=\max_id(i)>0$. A normalized potential is nondecreasing. For normalized $u$ define two potentials on $\{0,\dots,k\}$:
\[
\begin{aligned}
(\operatorname{pad}_Lu)(0)&=u(0), &(\operatorname{pad}_Lu)(i)&=u(i-1)\quad(1\le i\le k),\\
(\operatorname{pad}_Ru)(i)&=u(i)\quad(0\le i\le k-1), &(\operatorname{pad}_Ru)(k)&=u(k-1)+H.
\end{aligned}
\]
So $\operatorname{pad}_L$ puts an increment $0$ in front of the increments of $u$, and $\operatorname{pad}_R$ appends an increment $H$. Both potentials are normalized with the same $H$, and $\operatorname{pad}_L$ and $\operatorname{pad}_R$ commute. The normalized reflection of $u$ is $\bar u(i)=u(k-1-i)+Hi$. Its increments are $H-d(k-2-i)$, so it is normalized with the same $H$, $\bar{\bar u}$ is $u$ plus a constant, and $\operatorname{pad}_Ru$ is the normalized reflection of $\operatorname{pad}_L\bar u$ up to an additive constant. By \cref{lem:A}, $\cE(\bar u)$ is the image of $\cE(u)$ under $(i,j)\mapsto(k-1-i,k-1-j)$, and the reflection preserves the ray property, total support and facets.

We use two conditions on a normalized potential $u$:
\begin{itemize}[leftmargin=3.5em]
\item[$(\mathrm{P}_L)$] some $q$ has $d(q)=0$ and $(0,q)\in\cE(u)$;
\item[$(\mathrm{P}_R)$] some $q$ has $d(q)=H$ and $(q+1,k-1)\in\cE(u)$.
\end{itemize}
Condition $(\mathrm{P}_R)$ for $u$ is condition $(\mathrm{P}_L)$ for $\bar u$, and conversely. We call $\cE(u)$ \emph{connected} if the bipartite graph with row vertices and column vertices $0,\dots,k-1$, and an edge from row $i$ to column $j$ for every $(i,j)\in\cE(u)$, is connected.

\begin{lemma}\label{lem:strong}
If $\cE(u)$ is connected and has total support, then the digraph $D_\sigma$ of \cref{lem:D} is strongly connected for every tight permutation $\sigma$.
\end{lemma}

\begin{proof}
An arc $i\to i'$ of $D_\sigma$ with $i\ne i'$ comes from a tight edge $(i,j)$ with $j=\sigma(i')$. By total support some tight permutation $\pi$ has $\pi(i)=j$. For every row $a$ the edge $(a,\pi(a))$ is tight and gives the arc $a\to\sigma^{-1}(\pi(a))$, so the cycle of the permutation $\sigma^{-1}\pi$ through $i$ is a directed cycle of $D_\sigma$, and it contains the arc $i\to i'$. So every arc lies on a directed cycle, and every weakly connected component of $D_\sigma$ is strongly connected. Contracting each edge $\{\text{row }a,\ \text{column }\sigma(a)\}$ of the bipartite graph of $\cE(u)$ to a vertex $a$ turns an edge $(i,j)$ into the arc $i\to\sigma^{-1}(j)$. So the underlying graph of $D_\sigma$ is obtained from a connected graph by contracting edges, and it is connected.
\end{proof}

\begin{theorem}[Padding]\label{thm:pad}
Let $u$ be a normalized potential on $\{0,\dots,k-1\}$ with $k\ge3$, let $u'=\operatorname{pad}_Lu$, and let $v'$ be the infimal convolution of $u'$.
\begin{enumerate}[label=(\alph*)]
\item $v'(0)=v'(1)=2u(0)$ and $v'(t)=v(t-2)$ for $2\le t\le 2k$. The tight graph $\cE(u')$ consists of the pairs $(i+1,j+1)$ with $(i,j)\in\cE(u)$, the pairs $(0,0)$, $(0,1)$, $(1,0)$, and the pairs $(0,q+2)$, $(q+2,0)$ for every $q$ with $d(q)=0$ and $(0,q)\in\cE(u)$.
\item If $(\mathrm{P}_L)$ fails, $u'$ is not a ray.
\item Suppose that $u$ is a ray, that $\cE(u)$ is connected and has total support, and that $(\mathrm{P}_L)$ holds. Then $u'$ is a ray, $\cE(u')$ is connected and has total support, and $I(u')$ defines a facet of $\Ph(k+1)$. Moreover $u'$ satisfies $(\mathrm{P}_L)$, and it satisfies $(\mathrm{P}_R)$ if $u$ does.
\item The same statements hold for $\operatorname{pad}_Ru$ with the roles of $(\mathrm{P}_L)$ and $(\mathrm{P}_R)$ exchanged: if $(\mathrm{P}_R)$ fails, $\operatorname{pad}_Ru$ is not a ray; under the hypotheses of (c) with $(\mathrm{P}_R)$ in place of $(\mathrm{P}_L)$, $\operatorname{pad}_Ru$ is a ray with connected tight graph and total support, $I(\operatorname{pad}_Ru)$ defines a facet of $\Ph(k+1)$, and $\operatorname{pad}_Ru$ satisfies $(\mathrm{P}_R)$, and also $(\mathrm{P}_L)$ if $u$ does.
\end{enumerate}
\end{theorem}

\begin{proof}[Idea of the proof]
Part (a) is a direct comparison of costs, using that $u$ is nondecreasing. In (b), if no $q$ qualifies, the indicator of the index $0$ lies in $Z(u')$ and is not in $\spn\{\one,\iota,u'\}$. In (c), every $z\in Z(u')$ restricts on $\{1,\dots,k\}$ to a shifted element of $Z(u)$, and the tight pairs $(0,q+2)$ and $(1,q+1)$ given by $(\mathrm{P}_L)$ fix $z(0)$; total support follows from \cref{lem:strong,lem:D}, because the switching digraph of $u'$ contains a shifted copy of that of $u$ and arcs to and from the new index; then \cref{thm:R} gives the facet. Part (d) follows from (b) and (c) by the normalized reflection. The proof is in \cref{sec:SM-pad}.
\end{proof}

\begin{corollary}\label{cor:pad}
Let $u$ be a normalized ray whose tight graph is connected and has total support, and which satisfies $(\mathrm{P}_L)$ and $(\mathrm{P}_R)$. Then for all $a,b\ge0$ the potential $\operatorname{pad}_L^a\operatorname{pad}_R^bu$ satisfies the same hypotheses, and $I(\operatorname{pad}_L^a\operatorname{pad}_R^bu)$ defines a facet of $\Ph(k+a+b)$.
\end{corollary}

\begin{proof}
Induction on $a+b$, using \cref{thm:pad}(c),(d) and the commutation of the two paddings.
\end{proof}

Padding maps step potentials to step potentials, since it adds an increment $0$ or $H=1$ to a $0/1$ sequence, and it maps non-step potentials to non-step potentials; so the paddings of a non-step facet ray that satisfies the hypotheses of \cref{cor:pad} are non-step facets. These hypotheses are not automatic: by \cref{thm:pad}(b), a ray without $(\mathrm{P}_L)$ pads to a non-ray.

\emph{Cores.} Consider potentials up to positive scaling, and take facet rays (\cref{sec:tight}) normalized. A facet ray on $\{0,\dots,k-1\}$ is a \emph{core} if it is neither $\operatorname{pad}_Lw$ nor $\operatorname{pad}_Rw$ for a facet ray $w$ on $\{0,\dots,k-2\}$. Removing paddings as long as the result is a facet ray shows that every facet ray is $\operatorname{pad}_L^a\operatorname{pad}_R^bc$ for a core $c$ and some $a,b\ge0$. A facet ray that begins with a flat step need not be a padding: for $k=12$ the facet ray with increments $(0,2,0,2,1,1,1,0,1,1,1)$ is a core, and removing its first increment gives a potential that is not a ray.

\begin{corollary}[Counting by cores]\label{cor:cores}
Let $c_j$ be the number of cores on $\{0,\dots,j-1\}$. Suppose that every core on at most $k$ points satisfies the hypotheses of \cref{cor:pad}, and that every facet ray on $\{0,\dots,k-1\}$ has only one representation $\operatorname{pad}_L^a\operatorname{pad}_R^bc$ with a core $c$. Then the number of facets of $\Ph(k)$ of the form $I(u)$ is
\[
\sum_{j=3}^{k}c_j\,(k-j+1).
\]
\end{corollary}

\begin{proof}
By \cref{cor:pad}, $\operatorname{pad}_L^a\operatorname{pad}_R^bc$ with $a+b=k-j$ is a facet ray on $\{0,\dots,k-1\}$ with total support for every core $c$ on $j$ points, and there are $k-j+1$ such pairs $(a,b)$. Every facet of the form $I(u)$ is defined by a ray (\cref{rem:fan}), which after normalization is a facet ray and so has a representation, unique by assumption. By \cref{lem:distinct}, distinct facet rays with total support define distinct facets.
\end{proof}

Under the hypotheses of \cref{cor:cores}, let $F_j$ be the number of facets of $\Ph(j)$ of the form $I(u)$, with $F_1=F_2=0$. Then $F_j-F_{j-1}=\sum_{i=3}^{j}c_i$, and hence
\[
c_k=F_k-2F_{k-1}+F_{k-2}.
\]
For $3\le k\le12$ the hypotheses of \cref{cor:cores} hold by an exact check on the complete facet lists of \cref{sec:census} (computer-assisted), so the identity above holds for $k\le12$. For larger $k$ they can fail: the facet ray of \cref{ex:noTS} ($k=19$) does not have total support. We do not know whether total support of a ray implies the other hypotheses of \cref{cor:pad}.

\subsection{Refinement}\label{sec:refinement}
Padding adds one index at an end. The following operation multiplies the grid; by \cref{thm:refine} below it maps a facet ray whose tight graph is connected and has total support to a facet ray with the same properties. For a normalized potential $u$ on $\{0,\dots,K\}$, $K=k-1$, and an integer $m\ge2$, define potentials on $\{0,\dots,mK\}$ and on $\{0,\dots,mk-1\}$ by
\[
(\operatorname{spr}_mu)(i)=u(\lceil i/m\rceil),\qquad(\operatorname{ref}_mu)(i)=u(\lfloor i/m\rfloor).
\]
Each increment of $u$ appears in $\operatorname{spr}_mu$ followed by $m-1$ zeros, and $\operatorname{ref}_mu$ is $\operatorname{pad}_L^{m-1}\operatorname{spr}_mu$. Both are normalized with the same $H$. The potential $\operatorname{ref}_mu$ is the grid form on $(1/(mk))\Z$ of the step function $\tilde u$ on $[0,1]$ that equals $u(i)$ on $(i/k,(i+1)/k)$; so $I(u)$ on the grid $(1/k)\Z$ and $I(\operatorname{ref}_mu)$ on the grid $(1/(mk))\Z$ come from the same step function $\tilde u$.

In the rest of this subsection $u$ is a normalized potential on $\{0,\dots,K\}$, $m\ge2$, $u'=\operatorname{spr}_mu$ and $v'$ is the infimal convolution of $u'$. For $0\le i\le mK$ write $a_i=\lceil i/m\rceil$ and $r_i=ma_i-i\in\{0,\dots,m-1\}$; the \emph{block} $B_a=\{i:a_i=a\}$ is $\{0\}$ for $a=0$ and $\{m(a-1)+1,\dots,ma\}$ for $1\le a\le K$, and $u'=u(a)$ on $B_a$. A normalized potential is nondecreasing, and so is $v$: if $(i,j)$ attains $v(t+1)$, we may take $i\ge1$, and then $v(t)\le u(i-1)+u(j)\le v(t+1)$.

\begin{lemma}[Tight graph of $\operatorname{spr}_mu$]\label{lem:sprtight}
For $0\le t\le 2mK$, $v'(t)=v(\lceil t/m\rceil)$. A pair $(i,j)$ is tight for $u'$ if and only if $(a_i,a_j)\in\cE(u)$ and either $r_i+r_j\le m-1$, or $r_i+r_j\ge m$ and $v(a_i+a_j-1)=v(a_i+a_j)$.
\end{lemma}

\begin{proof}
Let $s=\lceil t/m\rceil$ and $\delta=ms-t\in\{0,\dots,m-1\}$. For $i+j=t$ we have $m(a_i+a_j)=t+r_i+r_j\in[t,t+2m-2]$, so $a_i+a_j=s$ if $r_i+r_j=\delta$, and $a_i+a_j=s+1$ if $r_i+r_j=\delta+m$; no other case occurs. The cost of $(i,j)$ is $u(a_i)+u(a_j)\ge v(a_i+a_j)\ge v(s)$, since $v$ is nondecreasing. Every pair $(a,b)$ with $a+b=s$ occurs among the block pairs $(a_i,a_j)$ of the anti-diagonal $t$: take $(i,j)=(ma-\delta,mb)$ if $a\ge1$, $(0,mb-\delta)$ if $a=0<b$, and $(0,0)$ if $s=0$ (then $t=0$). Hence $v'(t)=v(s)$. A pair with $a_i+a_j=s$ is tight iff $(a_i,a_j)\in\cE(u)$. A pair with $a_i+a_j=s+1$ is tight iff $u(a_i)+u(a_j)=v(s)$; since $u(a_i)+u(a_j)\ge v(s+1)\ge v(s)$, this holds iff $(a_i,a_j)\in\cE(u)$ and $v(s)=v(s+1)$. Finally $r_i+r_j\le m-1$ means $a_i+a_j=\lceil(i+j)/m\rceil$, and $r_i+r_j\ge m$ means $a_i+a_j=\lceil(i+j)/m\rceil+1$.
\end{proof}

For an anti-diagonal $s$ of $u$ let $P_s=\{p\ge1:(p,s-p)\in\cE(u)\}$, and let $\Theta(u)$ be the graph on $\{1,\dots,K\}$ in which each set $P_s$ is a clique.

\begin{lemma}\label{lem:blockgraph}
Let $u$ be a ray whose tight graph is connected and has total support. Then $\Theta(u)$ is connected.
\end{lemma}

\begin{proof}[Idea of the proof]
For each component $C$ of $\Theta(u)$ the vector $i\,\ind{i\in C}$ lies in $Z(u)$, and since $Z(u)$ has dimension $3$, there are at most two components. Two components force, through a tight permutation fixing $0$, a split of $\{1,\dots,K\}$ into two intervals that disconnects the bipartite graph of $\cE(u)$. The proof is in \cref{sec:SM-blockgraph}.
\end{proof}

\begin{lemma}[A flat level]\label{lem:flatlevel}
Let $u$ be normalized with a tight permutation $\sigma$ such that $\sigma(0)=0$. Then some $(a,b)\in\cE(u)$ with $a,b\ge1$ has $v(a+b-1)=v(a+b)$.
\end{lemma}

\begin{proof}
Let $q$ be an index with $u(q+1)=u(q)$. For every $j$ the pairs $(q,j)$ and $(q+1,j)$ cost the same, and $v(q+j)\le v(q+1+j)$; so if $(q,j)$ is tight, then $u(q+1)+u(j)=v(q+j)\le v(q+1+j)$ forces $(q+1,j)$ to be tight and $v(q+j)=v(q+1+j)$. If $q\ge1$, take $j=\sigma(q)\ge1$; if $q=0$, take $j=1$, since $(0,1)$ is the only pair on its anti-diagonal. Then $(a,b)=(q+1,j)$ works.
\end{proof}

\begin{theorem}[Grid refinement]\label{thm:refine}
Let $u$ be a normalized ray on $\{0,\dots,K\}$, $k=K+1\ge3$, whose tight graph is connected and has total support, and let $m\ge2$. Then $u'=\operatorname{spr}_mu$ is a normalized ray whose tight graph is connected and has total support, and $I(u')$ defines a facet of $\Ph(mK+1)$. Moreover $u'$ satisfies $(\mathrm{P}_L)$, and it satisfies $(\mathrm{P}_R)$ if $u$ does.
\end{theorem}

\begin{proof}[Structure of the proof]
Fix a tight permutation $\sigma$ of $u$ with $\sigma(0)=0$. For the ray property let $z\in Z(u')$. Tight pairs inside the lifted block pairs $(a,\sigma(a))$ make $z$ affine on every block with some slope $c_a$; pairs on the anti-diagonals $ms$ and $ms-1$ make the slopes of the blocks in one set $P_s$ equal, so by \cref{lem:blockgraph} all slopes are one number $c$; the values $z(ma)$ form an element $\alpha+\beta\iota+\gamma u$ of $Z(u)$; and a tight pair given by \cref{lem:flatlevel} forces $\beta=cm$, so $z=\alpha+c\iota+\gamma u'$. For total support, the lift $\pi$ of $\sigma$ that reverses offsets inside blocks is tight, and its switching digraph contains two copies of $D_\sigma$, strongly connected by \cref{lem:strong}, joined through every block; \cref{lem:D} applies. Connectivity lifts from $\cE(u)$, the facet is \cref{thm:R}, and $(\mathrm{P}_L)$, $(\mathrm{P}_R)$ are checked at explicit tight pairs. The proof is in \cref{sec:SM-refine}.
\end{proof}

\begin{corollary}[Finer grids]\label{cor:finer}
Under the hypotheses of \cref{thm:refine}, for every $m\ge2$ the potential $\operatorname{ref}_mu$ is a normalized ray with connected tight graph and total support, and $I(\operatorname{ref}_mu)$ defines a facet of $\Ph(mk)$. So the step function $\tilde u$ of such a potential, which gives a facet on the grid $(1/k)\Z$ by \cref{thm:R}, gives a facet on every grid $(1/(mk))\Z$. For $k\le12$ every facet ray satisfies the hypotheses (computer-assisted: by the exact check on the facet lists stated at the end of \cref{sec:padding}).
\end{corollary}

\begin{proof}
$\operatorname{ref}_mu=\operatorname{pad}_L^{m-1}u'$ with $u'=\operatorname{spr}_mu$, and $u'$ satisfies the hypotheses of \cref{thm:pad}(c) by \cref{thm:refine}. Each padding preserves them, so $m-1$ applications of \cref{thm:pad}(c) give the claim. For the last sentence let $w$ be a facet ray on $\{0,\dots,k-1\}$ with $k\le12$. Removing end paddings writes $w=\operatorname{pad}_L^a\operatorname{pad}_R^bc$ with a core $c$ and $a,b\ge0$ (\cref{sec:padding}, \cref{sec:census}). By the exact check stated at the end of \cref{sec:padding}, the core $c$ satisfies the hypotheses of \cref{cor:pad}, and \cref{cor:pad} shows that $w$ satisfies them too; in particular $w$ is a normalized ray whose tight graph is connected and has total support.
\end{proof}

Let $u$ satisfy the hypotheses of \cref{thm:refine}. If the first increment of $u$ is nonzero, then $\operatorname{spr}_mu$ begins with a nonzero increment and ends with the increment $0<H$, so it is neither $\operatorname{pad}_Lw$ nor $\operatorname{pad}_Rw$, and it is a core. Applied to the potentials $G_{a,b}$ with $a\le b+1$ of \cref{thm:height} below, which satisfy these hypotheses and begin with a nonzero increment, \cref{thm:refine} gives non-step cores of every height $h\ge2$ for infinitely many $k$ (take $a=1$).

Conversely, properties of $\operatorname{spr}_mu$ descend to $u$ as follows.

\begin{proposition}[Descent under refinement]\label{prop:descent}
Let $u$ be a normalized potential on $\{0,\dots,K\}$, $m\ge2$ and $u'=\operatorname{spr}_mu$.
\begin{enumerate}[label=(\alph*)]
\item If $\cE(u')$ has total support, so does $\cE(u)$. If $\cE(u')$ is connected, so is $\cE(u)$.
\item Let $\mathcal{S}$ be the set of $s\ge2$ with $v(s)=v(s-1)$ such that some $(a,b)\in\cE(u)$ with $a,b\ge1$ has $a+b=s$, and for $z\in Z(u)$ let $h_z(s)=z(a)+z(b)$ for any tight pair $(a,b)$ with $a+b=s$. Suppose that for every $z\in Z(u)$ the differences $h_z(s)-h_z(s-1)$, $s\in\mathcal{S}$, are all equal; this holds if $u$ has exactly one zero increment. If $u'$ is a ray, then $u$ is a ray.
\item The condition in (b) cannot simply be dropped: for the increments $(0,1,1,2,0)$ and $m=2$, $u'$ is a ray and $u$ is not; here $\cE(u')$ does not have total support.
\end{enumerate}
\end{proposition}

\begin{proof}[Idea of the proof]
For (a), a tight permutation of $u'$ through a lifted pair induces, after counting indices block by block, a doubly stochastic matrix supported on $\cE(u)$, and the Birkhoff--von Neumann theorem gives a tight permutation of $u$. For (b), the condition makes $z\mapsto z(a_i)-\beta_zr_i/m$ an injective linear map from $Z(u)$ into $Z(u')$. Part (c) is an exact computation. The proof is in \cref{sec:SM-descent}.
\end{proof}

Whether $u$ is a ray whenever $\operatorname{spr}_mu$ is a ray with total support is open (\cref{sec:conclusion}).

\subsection{Infinite families of facets from no step potential}\label{sec:nonstepfamilies}
Under the hypotheses of \cref{cor:pad}, padding produces facets from smaller ones. The following families give non-step facets for all large $k$. Write $K=k-1$, $L=k-2$, $\rho(i)=K-i$ and $\epsilon_a(i)=\ind{i=a}$.

\begin{proposition}[Non-step families]\label{thm:fam}
Each of the following potentials defines a facet of $\Ph(k)$.
\begin{enumerate}[label=(\alph*)]
\item For $a\ge2$ and $b\ge1$, with $n=a+b$ and $k=n+3$: $u_{a,b}(i)=\min(i,a)$ for $0\le i\le n$ and $u_{a,b}(n+1)=u_{a,b}(n+2)=a+2$. This gives $k-5$ potentials for each $k\ge6$.
\item For $k\ge9$: $\omega=\epsilon_1-\epsilon_L$, and its negative $-\omega$.
\item For $k\ge7$: $\nu=\epsilon_K-\epsilon_0-\epsilon_2$; for $k\ge9$: its negative $-\nu$.
\end{enumerate}
Moreover:
\begin{enumerate}[label=(\alph*),resume]
\item For $k\ge9$, the facets in (a)--(c), their reflections and the facets given by step potentials are pairwise distinct, up to $-\omega=\omega\circ\rho$ (the facet of $-\omega$ in (b) is the reflection of the facet of $\omega$).
\item Every potential in (a)--(c) satisfies the hypotheses of \cref{cor:pad}, so all its paddings $\operatorname{pad}_L^a\operatorname{pad}_R^b$ define facets.
\end{enumerate}
\end{proposition}

The negative $-\omega$ is the reflection of $\omega$, since $\omega\circ\rho=-\omega$; the negative $-\nu$ is not a reflection of $\nu$. The behavior for smaller $k$ is recorded in \cref{sec:SM-famproof}.

\begin{proof}[Structure of the proof]
For (a)--(c), in each case we compute $v$, exhibit an index $e$ that is tight with every index ($e=n$ for $u_{a,b}$, $e=L$ for $\omega$, $e=2$ for $\nu$, $e=K$ for $-\nu$) together with a tight permutation fixing $e$, so that \cref{lem:TS} gives total support, and show $Z(u)=\spn\{\one,\iota,u\}$ using two tight pairs on one anti-diagonal at a time; then \cref{thm:R} applies. Part (d) compares the patterns of \cref{lem:distinct}, all of which contain the entry $\tfrac12$ while step potentials have patterns in $\{0,1\}^{k-1}$. Part (e) checks $(\mathrm{P}_L)$ and $(\mathrm{P}_R)$ at explicit tight pairs. The proof is in \cref{sec:SM-famproof}, and the value tables and the computations of $Z(u)$ are in \cref{sec:SM-fam}.
\end{proof}

The written proofs need $k\ge9$ in (b) and for $-\nu$, and their formulas were also checked by exact computation (\cref{sec:SM-famproof}). Together with their reflections and paddings (\cref{cor:pad}), the potentials of \cref{thm:fam} give $44$, $80$, $130$ and $196$ of the $195$, $622$, $1{,}902$ and $5{,}748$ non-step facets for $k=9,10,11,12$; this count is exact, from the complete facet lists.

\begin{remark}[No divisor-periodic facets]\label{prop:P}
Let $k=nd$ with $n,d\ge2$. No facet of $\Ph(k)$ has a normal of the form $\gamma(t)=A+Bt+\psi(t\bmod d)$. If $\psi$ is not constant, the face exposed by such a normal has codimension at least $2n-1\ge3$. The proof, which uses permutations that preserve the blocks of $d$ consecutive indices, is in \cref{sec:SM-divisor}. Normals of this form with a period $d$ that does not divide $k$ do occur: at $k=6$, two of the four facets that come from no step potential have normals $A+Bt+\psi(t\bmod5)$.
\end{remark}

\subsection{Quadratic height}\label{sec:height}
The height of a facet ray (\cref{def:height}) is the largest multiplicity of a threshold in its tail-sum reading (\cref{rem:multiplicity}). The following family shows that it grows at least quadratically in $k$.

\begin{theorem}[Quadratic height]\label{thm:height}
For integers $a\ge1$ and $b\ge2$ let $G_{a,b}$ be the normalized potential on $\{0,\dots,a+2b\}$ with $G_{a,b}(0)=0$ and increments $(b^a,0^{b-1},(b-1)a+1,1^b)$, where $x^j$ denotes $j$ repetitions of $x$. Its tight graph is connected and has total support, it satisfies $(\mathrm{P}_L)$ and $(\mathrm{P}_R)$, and $\dim Z(G_{a,b})=3$ if $a\le b+1$ and $4$ otherwise. For $a\le b+1$, $G_{a,b}$ is a ray, $I(G_{a,b})$ defines a facet of $\Ph(a+2b+1)$, and $G_{a,b}$ is a core of height $(b-1)a+1$. Consequently, for every $k\ge3$ some facet ray of $\Ph(k)$ has height at least $\max\bigl(\lfloor(k-2)/2\rfloor,\ \lfloor(k-2)/3\rfloor^2\bigr)$, which is $\lfloor(k-2)/3\rfloor^2$ for $k\ge8$.
\end{theorem}

\begin{proof}
By \cref{lem:A} we may work with $w=G_{a,b}-\iota$. Put $p=b-1$, $e=a+b-1$, $n=a+2b$, $C=p(2a-1)$, $T=a+e$ and $S=e+1+n$, and the blocks $P=[0,a]$, $B=[a,e]$ and $V=[e+1,n]$. Then $w(i)=pi$ on $P$, $w(i)=ab-i$ on $B$ and $w=C$ on $V$. Comparing the least costs of the pairs of each block type on each anti-diagonal gives $v$ piecewise linear in $t$ and the list of tight pairs; this computation is in \cref{sec:SM-height}. Up to transposition the tight pairs are the pairs in $P\times P$ with $i+j\le a$, $(0,j)$ for $j\in B$, $(i,e)$ for $i\in P$, the pairs in $B\times B$ with $i+j\ge T$, $(e,j)$ for $j\in V$, the pairs in $V\times V$ with $i+j\ge S$, the pair $(a-1,n)$, and $(0,T)$ if and only if $T\le n$, that is, $a\le b+1$.

The index $e$ is tight with every index, and a tight permutation fixing $e$ exists, so \cref{lem:TS} gives total support and the tight graph is connected. Every $z\in Z(w)$ is affine on each of $P$, $B$ and $V$, which leaves five parameters, and two block types meet on one anti-diagonal only at $t=2e+1$ and, if $T\le n$, at $t=T$; the two resulting equations are independent. So $\dim Z(w)=3$ if $a\le b+1$ and $4$ otherwise, and in the first case \cref{thm:R} gives the facet. These verifications are in \cref{sec:SM-heightZ}.

The increments of $G_{a,b}$ are nonnegative integers with minimum $0$ and greatest common divisor $1$ (the last one is $1$), and the largest is $H=(b-1)a+1\ge2$. For $(\mathrm{P}_L)$: the increment at $a$ is $0$ and $(0,a)$ is tight. For $(\mathrm{P}_R)$: the increment at $e$ is $H$ and $(e+1,n)$ is tight. The first increment is $b\ne0$ and the last is $1<H$, so $G_{a,b}$ is neither $\operatorname{pad}_Lw'$ nor $\operatorname{pad}_Rw'$, and it is a core.

For the last statement, single steps give height $1$ for $k\ge3$ (\cref{thm:S}), which settles $k\le5$. For $k\ge6$ write $k=2b+2+j$ with $b=\lfloor(k-2)/2\rfloor\ge2$ and $j\in\{0,1\}$, and for $k\ge8$ write $k=3b+2+j$ with $b=\lfloor(k-2)/3\rfloor\ge2$ and $j\in\{0,1,2\}$. By \cref{cor:pad}, $\operatorname{pad}_L^jG_{1,b}$, respectively $\operatorname{pad}_L^jG_{b+1,b}$, is a facet ray of $\Ph(k)$ of height $b$, respectively $b^2$, since padding does not change the height. For $k\ge8$, $\lfloor(k-2)/3\rfloor^2\ge\lfloor(k-2)/2\rfloor$.
\end{proof}

The case $a=1$, $G_{1,b}$, has height $b$ on $2b+2$ points. The largest heights that occur for $k=4,\dots,12$ are $1,1,2,3,4,5,7,9,12$ (\cref{tab:shapes}); they equal the bound of \cref{thm:height} for $k=4,5,6,8,11$ and exceed it otherwise.

%% file: b5_steps.tex
\section{Step potentials}\label{sec:steps}

\subsection{Staircases}
A step set $J=\{j_1<\dots<j_r\}\subseteq\{1,\dots,k-1\}$ is recorded by its \emph{block lengths} $g=(g_0,\dots,g_r)$, a composition of $k$ with $g_0=j_1$, $g_l=j_{l+1}-j_l$ and $g_r=k-j_r$. The staircase $u=u(g)=u_J$ takes the value $l$ on block $l$; it has $r$ unit steps, at distinct positions, so here the number of unit steps counted with multiplicity is $r=|J|$. We call $l$ the \emph{level} of the indices in block $l$, and $p+q$ the level sum of a pair with levels $p$ and $q$. Write $S_l$ and $E_l$ for the first and last index of block $l$, and
\[
M(s)=\max_{p+q=s,\ 0\le p,q\le r}\,(E_p+E_q),\qquad 0\le s\le 2r.
\]
A block pair $(p,q)$ \emph{attains} $M(s)$ if $p+q=s$ and $E_p+E_q=M(s)$. For $r=0$ and $r=k-1$ the potential is affine and $I(u)$ holds with equality on $\Ph(k)$, so we assume $1\le r\le k-2$.

\begin{lemma}[Tight edges]\label{lem:stair}
The index sums reachable by pairs of level sum at most $s$ form the interval $[0,M(s)]$, and $M$ is strictly increasing. Hence $v(t)=\min\{s:M(s)\ge t\}$, and an edge $(i,j)$ with levels $p,q$ is tight if and only if $i+j>M(p+q-1)$ (no condition if $p=q=0$).
\end{lemma}

\begin{proof}
If $(i',j')$ has level sum at most $s$, so does $(i'-1,j')$, because $u$ is nondecreasing; so the reachable sums are down-closed, contain $0$, and have maximum $M(s)$. For $s<2r$ a maximizing pair of $M(s)$ has an index in a block below $r$, and moving it up one block raises the sum, so $M(s+1)>M(s)$. An edge is tight iff no pair of smaller level sum lies on its anti-diagonal.
\end{proof}

Write $x(>\tau)=\sum_{t>\tau}x(t)$.

\begin{corollary}[Staircases as tail inequalities]\label{cor:tails}
For a staircase, $v(t)=\#\{s<2r:\ t>M(s)\}$, so $I(u)$ reads
\[
\sum_{s=0}^{2r-1}x\bigl(>M(s)\bigr)\ \le\ 2\sum_{l=0}^{r}l\,g_l .
\]
Put $\eta_q=j_q/k$ for $1\le q\le r$, $\eta_{r+1}=1$, and $\xi_\ell=\max_{i+j=\ell-1,\ 0\le i,j\le r}(\eta_{i+1}+\eta_{j+1})$ for $1\le\ell\le2r$. Then $\xi_\ell=(M(\ell-1)+2)/k$ and $\sum_{q=1}^r(1-\eta_q)=\frac1k\sum_ll\,g_l$. So if $x\ge0$ lies in the affine hull and $S$ has mass $x(t)/k$ at $(t+1)/k$, then $I(u_J)$ divided by $k$ is
\begin{equation}\label{eq:tails}
\sum_{\ell=1}^{2r}\mathbb{P}(S\ge\xi_\ell)\ \le\ 2\sum_{q=1}^{r}(1-\eta_q).
\end{equation}
\end{corollary}

\begin{proof}
The first formula is \cref{lem:stair} with $M(2r)=2k-2$, and $2\sum_iu(i)=2\sum_ll\,g_l$. The first index of block $i+1$ is $E_i+1=j_{i+1}$, so $\eta_{i+1}=(E_i+1)/k$ for $0\le i\le r$, which gives $\xi_\ell=(M(\ell-1)+2)/k$. With $T=kS-1$, $S\ge\xi_\ell$ iff $T>M(\ell-1)$, and $\mathbb{P}(T=t)=x(t)/k$. Finally $\sum_q(k-j_q)=\sum_iu_J(i)=\sum_ll\,g_l$.
\end{proof}

By \cref{prop:C}, $x\in\Ph(k)$ exactly when $S=X+Y$ for a coupling of two standard uniforms, so a staircase inequality bounds the expected number of the levels $\xi_1\le\dots\le\xi_{2r}$ that $X+Y$ reaches, and \cref{thm:S,thm:S4,thm:T3} below say which of these bounds are facets. The companion paper \cite[Theorem~4.1]{CKcompanion} shows that \eqref{eq:tails} holds for all thresholds $0<\eta_1<\dots<\eta_r<1$ and that these bounds characterize the laws of $X+Y$; we do not use this. In the rest of this section we index these levels by the numbers $M(s)$.

\begin{lemma}[Corner triangles]\label{lem:corner}
For blocks $p,q\ge1$ put $\operatorname{tri}(p,q)=\max\bigl(0,\ E_p+E_q-M(p+q-1)\bigr)$. The tight edges between blocks $p$ and $q$ are exactly the pairs $(E_p-a,E_q-b)$ with $a,b\ge0$ and $a+b<\operatorname{tri}(p,q)$, and $\operatorname{tri}(p,q)\le\min(g_p,g_q)$. For $p=0<q$ define $\operatorname{tri}(0,q)$ by the same formula; the tight edges between blocks $0$ and $q$ are the pairs $(E_0-a,E_q-b)$ with $0\le a\le g_0-1$, $b\ge0$ and $a+b<\operatorname{tri}(0,q)$, and only $\operatorname{tri}(0,q)\le g_q$ is guaranteed (for $g=(1,3)$, $\operatorname{tri}(0,1)=3>g_0$). The case $q=0<p$ is symmetric, and block pair $(0,0)$ is complete.
\end{lemma}

\begin{proof}
By \cref{lem:stair} the tight edges of the pair are those with $i+j\ge t_*=M(p+q-1)+1$. For $p,q\ge1$, $M(p+q-1)\ge E_{p-1}+E_q$ and $M(p+q-1)\ge E_p+E_{q-1}$ give $t_*\ge\max(S_p+E_q,\ E_p+S_q)$. So $i+j\ge t_*$ forces $i\ge S_p$ and $j\ge S_q$, and the region is the corner triangle below $(E_p,E_q)$ with $E_p+E_q-t_*+1=\operatorname{tri}(p,q)\le\min(g_p,g_q)$ anti-diagonals. For $p=0<q$ only the second bound is available: it forces $j\ge S_q$ and gives $\operatorname{tri}(0,q)\le g_q$, while $i\ge S_0=0$ is the constraint $a\le g_0-1$.
\end{proof}

\begin{lemma}[The last anti-diagonal of a level]\label{lem:last}
The tight pairs on the anti-diagonal $M(s)$ are exactly the corners $(E_p,E_q)$ of the block pairs $(p,q)$ that attain $M(s)$. In particular, if a single block pair attains $M(s)$, its corner and the transpose are the only tight pairs there.
\end{lemma}

\begin{proof}
$v(M(s))=s$ by \cref{lem:stair}. A pair of levels $(p,q)$ with $p+q=s$ and index sum $M(s)$ has $i\le E_p$, $j\le E_q$ and $E_p+E_q\le M(s)$, so $i=E_p$, $j=E_q$, and $(p,q)$ attains $M(s)$.
\end{proof}

\begin{lemma}[Neighbors of block starts]\label{lem:starts}
For $0\le q\le r-1$, the tight neighbors of the block start $S_{q+1}=E_q+1$ are exactly the block ends $E_l$, $0\le l\le r$, such that $(q,l)$ attains $M(q+l)$.
\end{lemma}

\begin{proof}
Let $j$ lie in block $l$. The edge $(S_{q+1},j)$ has level sum $q+1+l\ge1$, so by \cref{lem:stair} it is tight iff $E_q+1+j>M(q+l)$. Since $j\le E_l$ and $E_q+E_l\le M(q+l)$, this holds iff $j=E_l$ and $E_q+E_l=M(q+l)$.
\end{proof}

So the tight graph of a staircase is determined by the sequence $E_0<\dots<E_r$ through the numbers $\operatorname{tri}(p,q)$, and the tight neighborhood of a set of block starts $\{S_{q+1}:q\in K\}$ is $\{E_l:l\in A(K)\}$, where $A(K)$ is the set of $l$ such that $(q,l)$ attains $M(q+l)$ for some $q\in K$.

\emph{Symmetries.} By \cref{lem:A}, the reflection sends the sign $-$ at $g$ to the sign $+$ at the reversed composition, and the sign $-$ at $J$ is the sign $+$ at $J^c$. Combining the two, the sign $+$ at $g^*$, the composition of $k-J^c$, exposes the mirror image $x(t)\mapsto x(2k-2-t)$ of the face exposed by the sign $+$ at $g$. The involution $g\mapsto g^*$ preserves facet status, total support and the ray property, and sends $r$ steps to $k-1-r$ steps. It does not preserve the facet itself: at $k=7$, $(2,1,1,3)$ and $(1,1,4,1)$ give mirror-image facets.

\subsection{Single, double and consecutive steps}

\begin{theorem}\label{thm:S}
For every $k\ge3$, each of the following potentials defines a facet of $\Ph(k)$, in both signs:
\begin{enumerate}[label=(\arabic*)]
\item single steps $u=\ind{i\ge a}$, $1\le a\le k-1$;
\item double steps $u=\ind{i\ge a}+\ind{i\ge b}$, $1\le a<b\le k-1$, $k\ge4$;
\item consecutive steps $u=\sum_{l=0}^{r-1}\ind{i\ge a+l}$, $1\le r\le k-2$, $1\le a\le k-r$.
\end{enumerate}
\end{theorem}

\begin{proof}
By \cref{lem:A} the reflection sends the sign $-$ at $J$ to the sign $+$ at $k-J$, which is again a single step, a double step or a consecutive run of the same length, so it suffices to treat the sign $+$. In each part we exhibit an index $e$ adjacent to every index and a tight permutation fixing $e$; then \cref{lem:TS} gives $\Es(u)=\cE(u)$, so $Z^*(u)=Z(u)$, and it remains to show $Z(u)=\spn\{\one,\iota,u\}$ (\cref{lem:U} or \cref{thm:R}).

For a single step at $a$ the index $e=a-1$ is universal, and every $z\in Z(u)$ is affine on each of the two blocks by \cref{lem:K}, with slopes equated by one pair of tight edges; double and consecutive steps follow the same pattern, with three blocks and with a ramp between two plateaus. The case checks for (1)--(3) are in \cref{sec:SM-S}.
\end{proof}

\begin{proposition}[Distinctness]\label{prop:Sprime}
Let $s\,u_J$ and $s'u_{J'}$ with $|J|=r$, $|J'|=r'$ and signs $s,s'\in\{\pm1\}$ be members of the families of \cref{thm:S}, and suppose $k-1>r+r'$. If they define the same facet, then $J=J'$ and $s=s'$.
\end{proposition}

\begin{proof}
Equal faces of $\Ph(k)$ have equal preimage faces in the Birkhoff polytope, hence the same tight permutations and the same allowed graph. Then $z=s'u_{J'}$ satisfies the defining equations of $Z^*(s\,u_J)$, which equals $Z(s\,u_J)$ by the total support shown in the proof of \cref{thm:S}, so $s'u_{J'}=A+D\iota+C\,s\,u_J$. Successive differences at $1,\dots,k-1$ give $s'\ind{J'}=D+C\,s\,\ind{J}$. Some position lies outside $J\cup J'$, so $D=0$, and then $J'=J$ and $s'=Cs$ with $C=\pm1$. If $s'=-s$, the tight graphs of $u$ and $-u$ coincide, so on every anti-diagonal a pair minimizing $u(i)+u(j)$ also maximizes it; then $u(i)+u(t-i)$ is constant along every anti-diagonal and $u$ is affine, a contradiction.
\end{proof}

So single steps give $2(k-1)$ distinct facets for $k\ge4$ and double steps give $2\binom{k-1}{2}$ distinct facets for $k\ge6$; at $k=5$ the twelve signed double steps give only six distinct facets, so the bound $k\ge6$ is sharp.

\subsection{Convex and concave rays}
For $x\ge0$ on the affine hull of $\Ph(k)$ let $T$ be a random variable with $\mathbb{P}(T=t)=x(t)/k$, $0\le t\le2k-2$, and let $U_k$ be uniform on $\{0,\dots,k-1\}$. Then $I(u)$ reads $\mathbb{E}\,v(T)\le2\,\mathbb{E}\,u(U_k)$; by \cref{prop:C}, $T=kS-1$ when $x$ comes from a lattice law of $S$. Convex and concave potentials give the classical convex-order conditions on $T$ \cite[Proposition~3.3]{CKcompanion}, and on the grid they account for exactly $2(k-2)$ facets: $k-2$ from the convex hinges defined below and $k-2$ from the concave hinges, their negatives. Here, for random variables $X,Y$ with finitely many values, $X\le_{\mathrm{cx}}Y$ (convex order) means $\mathbb{E}\,\phi(X)\le\mathbb{E}\,\phi(Y)$ for every convex $\phi:\R\to\R$, and for laws $\mu,\nu$ we write $\mu\le_{\mathrm{cx}}\nu$ in the same sense. For $1\le j\le k-2$ let $h_j(i)=(i-j)^+$, the \emph{hinge} at $j$. Parts (ii) and (iii) of \cref{prop:hinges}, \cref{lem:griddual,lem:convminorant} and \cref{cor:smallatoms} also appear in the companion paper \cite[Proposition~3.3, Lemmas~3.2 and~3.3, Corollary~3.2]{CKcompanion}; we include them with proofs so that this paper is self-contained.

\begin{proposition}[Hinges]\label{prop:hinges}
Let $k\ge3$.
\begin{enumerate}[label=(\roman*)]
\item A non-affine convex potential is a ray if and only if it equals $\alpha+\beta\iota+\gamma h_j$ with $\gamma>0$ and $1\le j\le k-2$; a non-affine concave potential is a ray if and only if the same holds with $\gamma<0$. All $2(k-2)$ of these rays give facets, and these facets are pairwise distinct.
\item For $u=h_j$, $v(t)=(t-2j)^+$, and $I(h_j)$ reads $\mathbb{E}(T-2j)^+\le\mathbb{E}(2U_k-2j)^+$. For $x\ge0$ on the affine hull, the $k-2$ inequalities $I(h_j)$ hold if and only if $T\le_{\mathrm{cx}}2U_k$.
\item For $u=-h_j$, $I(-h_j)$ reads $\mathbb{E}(T^\flat-j)^++\mathbb{E}(\kappa_k-j)^+\ge2\,\mathbb{E}(U_k-j)^+$, where $\kappa_k=(k-1)\ind{T\ge k}$ and $T^\flat=T-\kappa_k$. For $x\ge0$ on the affine hull, the $k-2$ inequalities $I(-h_j)$ hold if and only if $U_k\le_{\mathrm{cx}}\frac12\bigl(\law(T^\flat)+\law(\kappa_k)\bigr)$.
\end{enumerate}
\end{proposition}

\begin{proof}[Idea of the proof]
For a convex potential, a pair minimizes $u(i)+u(t-i)$ iff the increments are constant between its entries, so all convex potentials with the same runs of increments share a tight graph and lie in $Z(u)$ (for a concave potential the endpoint pair of each anti-diagonal always minimizes, and the other pairs do iff the increments are constant between the endpoints, with the same consequence); a ray therefore has at most two runs, that is, it is a hinge up to affine functions, and \cref{thm:S}(3) shows that hinges are rays that give facets in both signs. For (ii) and (iii), $v$ is computed from the balanced split, respectively the endpoint split, and convex order of integer-valued laws is decided by the hinges at integer points. Distinct hinge facets have second differences of $v$ that are not positive multiples of each other. The proof is in \cref{sec:SM-hingesproof} and \cref{sec:SM-hinges}.
\end{proof}

The convex hinges decide membership for X-rays with entries at most one. The proof uses two lemmas. The first is the finite duality for $\Ph(k)$.

\begin{lemma}[Grid duality]\label{lem:griddual}
Let $x\ge0$ on $\{0,\dots,2k-2\}$. Then $x\in\Ph(k)$ if and only if $\sum_t v(t)\,x(t)\le2\sum_iu(i)$ for every $u:\{0,\dots,k-1\}\to\R$.
\end{lemma}

\begin{proof}[Idea of the proof]
Necessity is the termwise bound $v(i+j)\le u(i)+u(j)$. For sufficiency, a separating functional $c$ pulls back to an assignment problem; an optimal dual pair, symmetrized, is a potential $u$ with $c\le v$ and the same objective value, and $x\ge0$ turns the separation into a violation of $I(u)$. The proof is in \cref{sec:SM-griddual}.
\end{proof}

For $u:\{0,\dots,k-1\}\to\R$ let $\varphi$ be its \emph{convex minorant}, the largest convex function on $[0,k-1]$ with $\varphi(i)\le u(i)$ for all $i$. It is the lower boundary of the convex hull of the points $(i,u(i))$, so it is piecewise linear with integer breakpoints, and $\varphi=u$ at the \emph{contact points} $0=a_0<a_1<\dots<a_N=k-1$, the endpoints and the breakpoints; $\varphi$ is affine on each $[a_{l-1},a_l]$.

\begin{lemma}[Convex minorant]\label{lem:convminorant}
Let $\varphi$ be the convex minorant of $u$ and $w=u-\varphi\ge0$ on $\{0,\dots,k-1\}$. Then
\[
\sum_{t=0}^{2k-2}\bigl(v(t)-2\varphi(t/2)\bigr)^+\ \le\ 2\sum_{i=0}^{k-1}w(i).
\]
\end{lemma}

\begin{proof}[Idea of the proof]
On a piece $[a,b]$ of $\varphi$, the pairs $(a,t-a)$ and $(b,t-b)$ are symmetric about $t/2$, so $v(t)\le2\varphi(t/2)+w(t-a)$ and $v(t)\le2\varphi(t/2)+w(t-b)$ on the two halves of $[2a,2b]$; summing over $t$ and over the pieces gives the bound. The proof is in \cref{sec:SM-convminorant}.
\end{proof}

\begin{corollary}[Small atoms]\label{cor:smallatoms}
Let $x$ lie in the affine hull of $\Ph(k)$ with $0\le x(t)\le1$ for all $t$. Then $x\in\Ph(k)$ if and only if $x$ satisfies the $k-2$ convex hinge inequalities $I(h_j)$ of \cref{prop:hinges}(ii); the concave hinge inequalities $I(-h_j)$ are then implied.
\end{corollary}

\begin{proof}
Necessity holds because each $I(h_j)$ is valid. For sufficiency, \cref{prop:hinges}(ii) gives $T\le_{\mathrm{cx}}2U_k$, and by \cref{lem:griddual} it suffices to show $\mathbb{E}\,v(T)\le2\,\mathbb{E}\,u(U_k)$ for every $u$. With $\varphi$ and $w$ as in \cref{lem:convminorant},
\begin{align*}
\mathbb{E}\,v(T)
&\le 2\,\mathbb{E}\,\varphi(T/2)+\mathbb{E}\bigl(v(T)-2\varphi(T/2)\bigr)^+\\
&\le 2\,\mathbb{E}\,\varphi(U_k)+\frac1k\sum_{t=0}^{2k-2}\bigl(v(t)-2\varphi(t/2)\bigr)^+\\
&\le 2\,\mathbb{E}\,\varphi(U_k)+\frac2k\sum_{i=0}^{k-1}w(i)\ =\ 2\,\mathbb{E}\,u(U_k).
\end{align*}
The first line is $y\le y^+$. In the second, $s\mapsto\varphi(s/2)$ is convex, so $\mathbb{E}\varphi(T/2)\le\mathbb{E}\varphi(2U_k/2)$ by $T\le_{\mathrm{cx}}2U_k$, and $\mathbb{P}(T=t)=x(t)/k\le1/k$ bounds the expectation of the nonnegative variable. The third line is \cref{lem:convminorant}.
\end{proof}

\begin{remark}\label{rem:LLYtest}
Li, Liu and Yao \cite[Corollary~4.4]{LiLiuYao2026} give a finite test for X-rays with all entries at most one through two families of linear inequalities; with the mean identity these reduce to the hinge inequalities $I(h_j)$, so their test is equivalent to \cref{cor:smallatoms}. \Cref{lem:griddual} is the content of \cite[Proposition~3.1]{LiLiuYao2026} used here. The bound $x\le1$ enters only in the second line of the chain, and it cannot be dropped: $x=(0,\frac32,0,\frac32,0)$ lies in the affine hull of $\Ph(3)$ and satisfies $I(h_1)$, but $x(0)=x(2)=x(4)=0$ forces $Z_{01}=Z_{21}=1$, so column $1$ has sum $2$.
\end{remark}

By \cref{prop:C}, \cref{cor:smallatoms} says that a lattice law on $\{1/k,\dots,(2k-1)/k\}$ with mean one and all atoms at most $1/k$ is the law of a sum of two standard uniforms if and only if $kS-1\le_{\mathrm{cx}}2U_k$, just as convex order decides membership for sums of three or more uniforms \cite{MaoWangWang2019}; this is the reading used in \cite[Corollary~3.2]{CKcompanion}. A facet is not implied by other valid inequalities, so every other facet is a constraint that the hinges do not imply: at $k=12$, $7{,}008$ of the $7{,}018$ ray facets, and $6{,}998$ if the concave hinges are also excluded.

\subsection{A criterion for any number of steps}

\begin{theorem}\label{thm:S4}
Let $g$ be a composition of $k$ with $r\ge1$ and $g_0\ge2$. Suppose
\begin{enumerate}[label=(\Alph*)]
\item[(U)] $E_p+E_q\le E_0+E_{p+q}$ whenever $p+q\le r-1$, and
\item[(R)] $E_{l-1}+E_l\ge M(2l-1)$ for every $1\le l\le r$.
\end{enumerate}
Then $u(g)$ is a ray with total support, hence a facet. By the symmetry $g\mapsto g^*$, the same conclusion holds whenever $g^*$ satisfies these hypotheses.
\end{theorem}

\begin{proof}
Put $e=E_0$. \emph{Universal vertex.} For $j$ in block $q\ge1$ the edge $(e,j)$ has level sum $q$, and (U) gives $M(q-1)=E_0+E_{q-1}<e+j$, so it is tight by \cref{lem:stair}; edges into block $0$ have cost $0$. \emph{A tight permutation fixing $e$.} Take the identity on block $0$ and the reversal on each block $l\ge1$. Its pairs have index sum $S_l+E_l=E_{l-1}+E_l+1>M(2l-1)$ by (R), so they are tight, and \cref{lem:TS} gives total support. \emph{Ray.} Let $z\in Z(u)$. For $j,j+1$ in one block $q$, the edges $(e,j)$ and $(e-1,j+1)$ are tight and lie on one anti-diagonal: for $q=0$ both lie in the complete block pair $(0,0)$, and for $q\ge1$ the second has level sum $q$ and index sum $E_0+j>E_0+E_{q-1}$ (here $g_0\ge2$ is used). So $z(j+1)-z(j)=z(e)-z(e-1)=:D$ in every block. Write $z(i)=A+Di+\omega(\text{level of }i)$ with $\omega(0)=0$. For $2\le q\le r$ the edges $(E_0,S_q)$ and $(S_1,E_{q-1})$ have levels $(0,q)$ and $(1,q-1)$ and the same index sum $E_0+E_{q-1}+1>M(q-1)$, so both are tight and $\omega(q)=\omega(1)+\omega(q-1)$. Hence $\omega(q)=q\,\omega(1)$ and $z\in\spn\{\one,\iota,u\}$.
\end{proof}

With the symmetry, \cref{thm:S4} certifies 764 of the 2,811 compositions $g$ with $3\le k\le12$ for which $u(g)$ with sign $+$ defines a facet, including 308 for which neither $J$ nor $J^c$ belongs to the families of \cref{thm:S}.

\subsection{Certificates of non-facets}

\begin{lemma}[One empty anti-diagonal]\label{lem:C1}
Let $u=u(g)$ be a staircase with $1\le r\le k-2$. If some anti-diagonal $t$ carries no edge of any tight permutation, then $I(u)$ is not a facet.
\end{lemma}

\begin{proof}
This is the case of \cref{thm:empty}(b) for staircases.
\end{proof}

A block start $S_l$ with a single tight neighbor gives such anti-diagonals in two cases. By \cref{lem:starts} the neighbor is a block end $E_q$, with $\operatorname{tri}(l,q)=g_l$ by \cref{lem:corner}. Every tight permutation maps $S_l$ to $E_q$, and since the tight graph is symmetric and inverses of tight permutations are tight, it also maps $E_q$ to $S_l$. Row $S_0=0$ is tight with all of block $0$, and with $E_1$ when $g_0=1$, so $l\ge1$.

\begin{proposition}\label{prop:H1}
If some block start $S_l$ has the single tight neighbor $E_0$ or $E_r$, then $I(u)$ is not a facet.
\end{proposition}

\begin{proof}
Partner $E_0$: by \cref{lem:last} with $s=0$, the only tight pair on the anti-diagonal $2E_0=M(0)$ is $(E_0,E_0)$, which is never used because $E_0$ is forced to $S_l\ne E_0$. Partner $E_r=k-1$: the only pair on the anti-diagonal $2k-2$ is $(k-1,k-1)$. Here $S_l\ne E_r$: otherwise $l=r$, $g_r=1$, and $E_r$ would also be tight with $E_{r-1}$, because $E_{r-1}+E_r=M(2r-1)>M(2r-2)$ by \cref{lem:stair}, contradicting uniqueness. \Cref{lem:C1} finishes both cases.
\end{proof}

\subsection{Step relations}
The following lemma gives a sufficient condition for the ray property of a staircase in terms of relations among its steps, for any number of steps; it is not necessary (see the examples in \cref{sec:s_census-proofs}). An index $i$ of block $l$ is written $(l,a)$, where $a=E_l-i$ is its offset from the block end. For $z\in Z(u)$, put $w_z(t)=z(i)+z(j)$ for any tight edge $(i,j)$ with $i+j=t$; membership in $Z(u)$ makes this well defined.

\begin{lemma}[Step relations]\label{lem:square}
Let $u=u(g)$ be a staircase, $z\in Z(u)$ and $d(i)=z(i+1)-z(i)$. Call $i$ a within-block step if $i$ and $i+1$ lie in one block, and a boundary step if $i=E_p$ with $p<r$.
\begin{enumerate}[label=(\roman*)]
\item Within-block steps $i=(p,a)$ and $j=(q,b)$ satisfy $d(i)=d(j)$ whenever $a+b\le \operatorname{tri}(p,q)$ (always if $p=q=0$).
\item Boundary steps $E_p$ and $E_q$ satisfy $d(E_p)=d(E_q)$ whenever $p$ and $q$ both occur in block pairs attaining $M(s)$ for one level sum $s$; in particular whenever $(p,q)$ attains $M(p+q)$. The equivalence classes of boundary steps generated by these relations are the \emph{classes} of $u$.
\item If all within-block steps are related by chains of (i) and all boundary steps by chains of (ii), then $u$ is a ray.
\item If the boundary steps of $u(g)$ form one class under (ii), and so do the boundary steps of $u(g^*)$, then $u(g)$ is a ray.
\end{enumerate}
\end{lemma}

\begin{proof}[Idea of the proof]
(i) and (ii) compare two tight edges on one anti-diagonal, given by \cref{lem:corner} and by \cref{lem:last,lem:stair}. Under (iii), $d$ takes one value on within-block steps and one on boundary steps, so $z\in\spn\{\one,\iota,u\}$. For (iv), $u(g^*)=u(g)\circ\rho+\iota+c$, and the boundary steps of $g^*$ correspond to the within-block steps of $g$, so (ii) for $g^*$ supplies the relations that (iii) needs. The proof is in \cref{sec:SM-squareproof}.
\end{proof}

Relation (ii) is needed in its general form, not only for pairs $(p,q)$ that attain $M(p+q)$; an example with $k=22$ is in \cref{sec:SM-square}.

\subsection{Three steps}
In the offset notation, by \cref{lem:corner}, an edge joining $(l,a)$ and $(q,b)$ is tight iff $a+b<\operatorname{tri}(l,q)$; block pair $(0,0)$ is complete, and block $0$ only adds the bound $a\le g_0-1$.

\begin{theorem}\label{thm:T3}
Let $k\ge5$ and let $g=(g_0,g_1,g_2,g_3)$ be a composition of $k$. The staircase $u(g)$ with sign $+$ defines a facet of $\Ph(k)$ if and only if it is not the case that $g_1,g_2,g_3$ are pairwise distinct with $g_1<g_3$. The value of $g_0$ plays no role.
\end{theorem}

Equivalently, the sign $-$ at $g$ fails to give a facet iff $g_2,g_1,g_0$ are pairwise distinct with $g_2<g_0$.

\begin{proof}
\emph{Triangle sizes.} For three steps, the pairs attaining $M(2)$, $M(3)$, $M(4)$ are decided by the signs of $g_2-g_1$, $g_3-g_1$ and $g_3-g_2$, and \cref{lem:corner} gives every $\operatorname{tri}(p,q)$ as an explicit expression in $g_1,g_2,g_3$ (\cref{sec:SM-T3fail}).

\emph{Failing direction.} Suppose $g_1,g_2,g_3$ are pairwise distinct and $g_1<g_3$. There are three cases: (A) $g_1<g_2<g_3$, (B) $g_1<g_3<g_2$ and (C) $g_2<g_1<g_3$. In case A the rows $S_2$ and $S_3$ both have exactly the tight neighbors $E_0$ and $E_3$; in case B the row $S_2$ has the single tight neighbor $E_0$; in case C the row $S_3$ has the single tight neighbor $E_3$. In each case every tight permutation is forced on these indices, and the anti-diagonal $2k-2$ (cases A and C) or $2E_0$ (case B), whose only tight pair is $(E_3,E_3)$ or $(E_0,E_0)$, carries no edge of a tight permutation; \cref{lem:C1} then rules out a facet. The neighbor computations are in \cref{sec:SM-T3fail}.

\emph{Facet direction.} Otherwise (I) $g_1\ge g_3$, or (II) $g_1=g_2<g_3$, or (III) $g_1<g_2=g_3$. In each case a tight involution that is the identity on block $0$ and reverses or pairs the other blocks has a strongly connected switching digraph, so \cref{lem:D} gives total support, and the boundary and within-block steps are each related by chains of \cref{lem:square}(i), (ii), so $u(g)$ is a ray; \cref{thm:R} gives the facet. The case checks are in \cref{sec:SM-T3}.
\end{proof}

\begin{corollary}\label{cor:T3}
\begin{enumerate}[label=(\alph*)]
\item Equal-gap triples $J=\{a,a+d,a+2d\}$ give facets in both signs, and triples with $\max J=k-1$ give facets with sign $+$.
\item For $k\ge7$, the triple $J=\{1,2,k-2\}$ with sign $+$ does not give a facet.
\item Every signed step potential with $r$ unit steps (counted with multiplicity), $1\le r\le k-2$ and $r\le3$ or $r\ge k-4$, is classified by \cref{thm:S,thm:T3} and the symmetry $g\mapsto g^*$.
\end{enumerate}
\end{corollary}

\begin{proof}
(a) Equal gaps mean $g_1=g_2$, and $\max J=k-1$ means $g_3=1\le g_1$; the sign $-$ of an equal-gap triple is the sign $+$ of the reversed triple, which again has equal gaps. (b) Here $g=(1,1,k-4,2)$, which is case B for $k\ge7$. (c) $g\mapsto g^*$ sends $r$ steps to $k-1-r$ steps and preserves facet status.
\end{proof}

Every step of the proof of \cref{thm:T3} was also checked by brute force on all three-step potentials with $k\le30$, and the statement itself on all 91,389 three-step potentials with $k\le40$ (\cref{app:computation}).

\subsection{Nonincreasing tails}
Case (I) of \cref{thm:T3} contains the three-step staircases with $g_1\ge g_2\ge g_3$. This case extends to any number of steps.

\begin{theorem}[Nonincreasing tails]\label{thm:dec}
Let $g=(g_0,\dots,g_r)$ be a composition of $k$ with $1\le r\le k-2$ and $g_1\ge g_2\ge\dots\ge g_r$; $g_0$ is arbitrary. Then $u(g)$ is a ray whose tight graph has total support, and $u(g)$ with sign $+$ defines a facet of $\Ph(k)$. Distinct such compositions give distinct facets.
\end{theorem}

\begin{proof}
\emph{The numbers $M(s)$.} Since the tail $g_1,\dots,g_r$ is nonincreasing, the sequence $E_0,\dots,E_r$ is concave, so $M(s)=E_{\lfloor s/2\rfloor}+E_{\lceil s/2\rceil}$, and \cref{lem:corner} gives $\operatorname{tri}(l,l)=\operatorname{tri}(l-1,l)=g_l$ for $1\le l\le r$ (\cref{sec:SM-dec}). So $(l,a)$ and $(l,b)$ are tight iff $a+b<g_l$, and $(l-1,a)$ and $(l,b)$ are tight iff $a+b<g_l$ (with $a\le g_0-1$ if $l=1$).

\emph{Total support.} The involution $\sigma$ that is the identity on block $0$ and reverses every block $l\ge1$ is tight. In its switching digraph every index reaches $E_0$ through the block starts $S_l\to S_{l-1}$, and $E_0$ reaches every index through the block ends $E_{l-1}$, which are tight with all of block $l$; so \cref{lem:D} gives total support.

\emph{Ray.} The pairs $(l-1,l)$ relate all boundary steps by \cref{lem:square}(ii), and since the tail $g_1,\dots,g_r$ is nonincreasing the blocks with at least two elements are consecutive, so (i) relates all within-block steps; \cref{lem:square}(iii) and \cref{thm:R} give a ray and a facet. The details of both parts are in \cref{sec:SM-dec}.

\emph{Distinctness.} The pattern of $u_J$ in the sense of \cref{lem:distinct} is $p(u_J)(i)=\ind{i+1\in J}$, since $1\le r\le k-2$. So different step sets have different patterns, and by \cref{lem:distinct} they define different facets.
\end{proof}

\begin{corollary}\label{cor:dec}
For $k\ge3$, $\Ph(k)$ has at least $\sum_{n=1}^{k-1}\mathrm{p}(n)-1$ facets, where $\mathrm{p}(n)$ is the number of partitions of $n$. In particular the number of facets of $\Ph(k)$ grows faster than any polynomial in $k$.
\end{corollary}

\begin{proof}
A composition as in \cref{thm:dec} is a choice of $g_0\in\{1,\dots,k-1\}$ and a partition $g_1\ge\dots\ge g_r$ of $n=k-g_0$; the only choice with $r>k-2$ is $g_0=1$ with $g_1=\dots=g_{k-1}=1$. For a fixed $c\ge2$, take $\lambda_i$ parts equal to $i$ for $2\le i\le c$, with $0\le\lambda_i\le n/(c(c-1))$; these parts sum to at most $n$, and parts equal to $1$ fill up the rest. This gives distinct partitions of $n$, so $\mathrm{p}(n)\ge(n/c^2)^{c-1}$, and the sum in the statement is at least $\mathrm{p}(k-1)$. (Hardy and Ramanujan \cite{HardyRamanujan1918} give $\log\mathrm{p}(n)\sim\pi\sqrt{2n/3}$.)
\end{proof}

For $k=12$ the bound is $193$, and the family of \cref{thm:dec} has $2{,}085$ members at $k=20$. The symmetry $g\mapsto g^*$ maps these facets to facets as well. The reversed condition does not suffice: $g=(1,1,2,3)$ at $k=7$ (case A of \cref{thm:T3}) gives no facet. We also checked \cref{thm:dec} by the exact facet test for all $k\le14$ ($371$ compositions at $k=14$).

%% file: b6_census.tex
\section{The facet census and total support}\label{sec:census}

\subsection{Exact facet counts}
By \cref{rem:fan,lem:coord}, for $k\ge4$ the facets of $\Ph(k)$ are the $2k-1$ coordinate inequalities and the inequalities $I(u)$ for the rays $u$ that give facets. For each anti-diagonal $t$ let $Q_t$ be the simplex with vertices $e_i+e_j$, $i+j=t$, $i\le j$. The pairs attaining $v(t)$ span the face $Q_t^u$ of $Q_t$ minimized by $u$, so the fan of \cite{LiLiuYao2026} is the common refinement of the normal fans of the $Q_t$, that is, the normal fan of the Minkowski sum $Q=\sum_tQ_t$ \cite[Lemma~2.1.5]{GritzmannSturmfels1993}, and its rays are the facet normals of $Q$; it is rational because $Q$ has integer vertices. This is an instance of the dictionary between tropical hyperplane arrangements and Minkowski sums of simplices \cite{ArdilaDevelin2009,DevelinSturmfels2004,FinkRincon2015} (\cref{sec:s_census-dict}).

We walk the facet graph of $Q$, which is connected, as in Fukuda's output-sensitive algorithm \cite{Fukuda2004}. The ridges at a facet $u$ are the facets of the face $Q^u=\sum_tQ_t^u$ minimized by $u$ \cite[Proposition~2.1]{Fukuda2004}, again a Minkowski sum of simplices; we list them as the flats of corank one of its difference matroid (the vector matroid of differences of two vertices of one summand) that pass a face-rank test, and cross each by an exact line search. For $k\ge4$, \cref{prop:complete} proves that this lists every ray, starting from any certified facet, and \cref{app:computation} gives the details and the independent recomputation. A non-facet is certified by the absence of a tight permutation, by an anti-diagonal carrying no allowed edge (\cref{thm:empty}), or by an exact rank deficiency; a facet is certified by a rank computed modulo a prime, which can only undercount the rational rank, with an exact rational recomputation whenever it falls short. The growth of the output and the costs are discussed in \cref{sec:s_census-cross}.

\begin{table}[ht]
\centering
\caption{Exact facet counts. The count for $k=5$ is in \cite[Example~3.5]{LiLiuYao2026}; the values for $k\ge6$ are new, and for $k\le8$ they agree with independent floating-point vertex--facet enumerations whose rows were each verified in exact arithmetic. Neither sequence is in the OEIS (searched October 3, 2026: the full facet sequence, its terms from $k=5$, and the ray counts from $k=5$ and from $k=6$).}
\label{tab:counts}
\begin{tabular}{@{}lrrrrrrrrrrr@{}}
\toprule
$k$ & 2 & 3 & 4 & 5 & 6 & 7 & 8 & 9 & 10 & 11 & 12\\
\midrule
facets of $\Ph(k)$ & 2 & 6 & 13 & 23 & 45 & 85 & 187 & 426 & 1033 & 2633 & 7041\\
rays of the fan & & & & 16 & 42 & 110 & 338 & 1146 & 4370 & 18515 & 89321\\
\bottomrule
\end{tabular}
\end{table}

For $4\le k\le12$ the facet count is $2k-1$ plus the number of facet rays. For $k=2$ and $k=3$ the coordinate inequalities define only $2$ and $4$ distinct facets ($x(0)\ge0$ and $x(2)\ge0$ define the same facet of the segment $\Ph(2)$, and $x(2)\ge0$ is not a facet of $\Ph(3)$, \cref{sec:SM-prelim-ex}), and there are $0$ and $2$ facet rays. No inequality $I(u)$ defines a coordinate facet, for any $k$ (\cref{thm:empty}(a)). For $k\le12$ every facet ray has total support, so distinct facet rays define distinct facets (\cref{lem:distinct}); both facts were checked directly for $5\le k\le12$. For $k\le4$ they follow from a direct computation of the facets of $\Ph(k)$: $\Ph(2)$ has no facet ray, and for $k=3,4$ every facet ray is a single step in one of the two signs, which has total support by the proof of \cref{thm:S}. The injectivity fails in general: at $k=19$ two distinct facet rays define the same facet (\cref{rem:samefacet}), and at $k=19f$ there are $2^f$ (\cref{prop:cubical}), so for larger $k$ the number of facets of the form $I(u)$ can be smaller than the number of facet rays. By \cref{thm:fiber}(d), injectivity holds for a given $k$ exactly when every facet ray of $\Ph(k)$ has total support.

\begin{table}[ht]
\centering
\caption{Facets by the smallest step family that produces them. A step set and its complement give the same facet with opposite signs (\cref{lem:A}), so four-step facets first appear at $k=9$ and five-step facets at $k=11$. ``Other'' counts facets produced by no step potential.}
\label{tab:census}
\begin{tabular}{@{}rrrrrrrrr@{}}
\toprule
$k$ & facets & coordinate & single & double & triple & four-step & five-step & other\\
\midrule
5 & 23 & 9 & 8 & 6 & 0 & 0 & 0 & 0\\
6 & 45 & 11 & 10 & 20 & 0 & 0 & 0 & 4\\
7 & 85 & 13 & 12 & 30 & 17 & 0 & 0 & 13\\
8 & 187 & 15 & 14 & 42 & 58 & 0 & 0 & 58\\
9 & 426 & 17 & 16 & 56 & 88 & 54 & 0 & 195\\
10 & 1033 & 19 & 18 & 72 & 126 & 176 & 0 & 622\\
11 & 2633 & 21 & 20 & 90 & 174 & 276 & 150 & 1902\\
12 & 7041 & 23 & 22 & 110 & 234 & 404 & 500 & 5748\\
\bottomrule
\end{tabular}
\end{table}

The facets produced by no step potential are 72\% of all facets at $k=11$ and 82\% at $k=12$; for $k\le12$ each of them is given by a ray with total support, and infinite families of them are given in \cref{sec:nonstepfamilies}.

\subsection{Cores and heights}
Every facet ray with $k\le12$ arises from exactly one core (\cref{sec:padding}) by end paddings (exact check). A padding of a step potential is a step potential, and a padding of any other potential is not, so the facets produced by no step potential come from cores of the same kind. \Cref{tab:cores} counts them: padding explains 3,182 of the 5,748 non-step facets at $k=12$ (55\%), and the remaining 2,566 are cores, for which we have no classification.

\begin{problem}\label{prob:C}
(a) Classify the cores. (b) Show that every normalized ray with total support has a connected tight graph and satisfies $(\mathrm{P}_L)$ and $(\mathrm{P}_R)$.
\end{problem}

Lifts with a two-point prefix account for part of \cref{prob:C}(a): by a rule that we verified by exact computation but have not proved, applied to the facet rays at $k=10$ together with the mirror rule, they give $962$ of the $2{,}808$ cores at $k=12$, among them $880$ of the $2{,}566$ non-step cores (\cref{sec:s_census-lifts}).

The facet rays are counted by height (\cref{def:height}) in \cref{tab:shapes}; height $1$ means a step potential. Height $2$ accounts for all non-step facets up to $k=6$ but for only 2,780 of 5,748 (48\%) at $k=12$. The largest height is unbounded: by \cref{thm:height} it is at least $\lfloor(k-2)/3\rfloor^2$ for $k\ge8$, so no bound on the multiplicity of thresholds suffices for all $k$.

\subsection{Total support for step potentials}

\begin{conjecture}\label{conj:T}
A staircase $u(g)$ with $1\le r\le k-2$, ray or not, defines a facet of $\Ph(k)$ if and only if its tight graph has total support.
\end{conjecture}

The analogous statement for arbitrary rays is false: for a ray, total support suffices for a facet by \cref{thm:R}, but the facet ray of \cref{ex:noTS} ($k=19$, height two) does not have total support. For staircases, \cref{thm:empty,cor:facetbip} show that a facet staircase without total support would have an allowed graph that meets every anti-diagonal and has a bipartite component; by \cref{prop:stairbip} below, the bipartite component is present for every staircase with a tight permutation but without total support. By \cref{thm:R}, \cref{conj:T} follows from two statements.

\begin{problem}[Facet implies total support]\label{prob:A}
If a staircase defines a facet, its tight graph has total support.
\end{problem}

\begin{problem}[Total support implies ray]\label{prob:B}
If the tight graph of a staircase has total support, the staircase is a ray.
\end{problem}

Both hold for at most three and at least $k-4$ steps (\cref{thm:S,thm:T3} and the symmetry $g\mapsto g^*$). Below we prove \cref{prob:A} for $r\le7$ and $r\ge k-8$ and for several classes of tails of every length, and give sufficient conditions for both problems in terms of the sequence $E_0<\dots<E_r$ alone. Write $N(Y)$ for the set of tight neighbors of a set $Y$ of indices; the tight graph is symmetric. Recall that $A(K)$, for $K\subseteq\{0,\dots,r-1\}$, is the set of $l$ such that $(q,l)$ attains $M(q+l)$ for some $q\in K$, and that $N(\{S_{q+1}:q\in K\})=\{E_l:l\in A(K)\}$ (\cref{lem:starts}). We write $A(q)$ for $A(\{q\})$.

\begin{theorem}\label{thm:H}
The tight graph of a staircase $u(g)$ with $1\le r\le k-2$ fails to have total support if and only if some nonempty set $X\subseteq\{S_1,\dots,S_r\}$ of block starts satisfies $|N(X)|\le|X|$.
\end{theorem}

The start $S_0=0$ is excluded because the block pair $(0,0)$ is complete: a set $W$ of rows with $|N(W)|\le|W|$ can be replaced by one that misses block $0$ (Step 2 of the proof). By the remark after \cref{lem:starts} the condition depends on the sequence $E$ alone.

\begin{proof}
We use three facts. (F1) If $|N(X)|=|X|$, every tight permutation $\sigma$ maps $X$ onto $N(X)$, and since $\sigma^{-1}$ is tight, it maps $N(X)$ onto $X$. (F2) The start $S_{q+1}=E_q+1$ equals a block end $E_d$ only if $d=q+1$ and $g_d=1$; in particular $E_0$ is none of $S_1,\dots,S_r$. (F3) Write an index of block $p$ as $E_p-a$ with offset $0\le a\le g_p-1$. For $p\ge1$, by \cref{lem:corner} and symmetry, the row $E_p-a$ is tight with the column $E_q-b$ of block $q$ iff $a+b<\operatorname{tri}(p,q)$. Hence the neighborhoods of the rows of block $p\ge1$ increase with the index. Moreover $M(p+l-1)\ge E_{p-1}+E_l$ gives $\operatorname{tri}(p,l)\le g_p$, with equality iff $l\in A(p-1)$, and \cref{lem:corner} gives $\operatorname{tri}(p,l)\le g_l$ for $l\ge1$; so $l\in A(p-1)$ and $l\ge1$ imply $g_l\ge g_p$.

\emph{If.} Let $X=\{S_{q+1}:q\in K\}$ be nonempty with $|N(X)|\le|X|$. If $|N(X)|<|X|$, Hall's theorem excludes a tight permutation. Let $|N(X)|=|X|$. If there is no tight permutation, total support fails by definition; so let one exist, and let $q_0=\min K$. The start $S_{q_0+1}$ has a tight neighbor, so by \cref{lem:starts} some $c\in A(q_0)$ exists, and by \cref{lem:last} the edge $(E_c,E_{q_0})$ is tight. Since $E_c\in N(X)$, every tight permutation maps $E_c$ into $X$ by (F1). By (F2) and the minimality of $q_0$, $E_{q_0}\notin X$. So $(E_c,E_{q_0})$ lies in no tight permutation.

\emph{Only if.} \emph{Step 1.} If there is no tight permutation, Hall's theorem gives a nonempty set $W$ of rows with $|N(W)|<|W|$. Otherwise let $\sigma$ be tight; by \cref{lem:D} the switching digraph $D_\sigma$ is not strongly connected, so it has a nonempty proper set $W$ of vertices closed under its arcs, that is, $\sigma^{-1}(N(W))\subseteq W$; with $\sigma(W)\subseteq N(W)$ this gives $|N(W)|=|W|$. \emph{Step 2.} Suppose $W$ meets block $0$. Then all of block $0$ lies in $N(W)$, since block pair $(0,0)$ is complete, so $W'=\{0,\dots,k-1\}\setminus N(W)$ misses block $0$. By symmetry no index of $W$ is tight with an index of $W'$, so $|N(W')|\le k-|W|$. If $|N(W)|<|W|$, then $|W'|=k-|N(W)|>|N(W')|$; if $|N(W)|=|W|$ and $W$ is proper, then $|W'|=k-|W|\ge\max(1,|N(W')|)$. In both cases we replace $W$ by $W'$. So we may assume that $W$ is nonempty, misses block $0$, and has $|N(W)|\le|W|$. \emph{Step 3.} For $1\le p\le r$ let $y_p=\max\{t+1:S_p+t\in W\}$, or $y_p=0$ if $W$ misses block $p$, and let $Y=\{S_p+t:1\le p\le r,\ 0\le t<y_p\}\supseteq W$. By (F3) every row of $Y$ in block $p$ has its neighbors among those of $S_p+y_p-1\in W$, so $N(Y)=N(W)$ and $|N(Y)|\le|Y|$. \emph{Step 4.} Let $T=\max_py_p\ge1$, and for $1\le h\le T$ let $P_h=\{p:y_p\ge h\}\ne\emptyset$, $K_h=\{p-1:p\in P_h\}$ and $X_h=\{S_p:p\in P_h\}$. Then $|Y|=\sum_h|P_h|$ and $|N(X_h)|=|A(K_h)|$. Fix $h$ and $l\in A(K_h)$, and choose $p\in P_h$ with $l\in A(p-1)$; the row $S_p+h-1$ lies in $Y$ and has offset $g_p-h$. If $l\ge1$, or $l=0$ and $h\le g_0$, the index $E_l-(h-1)$ lies in block $l$ (for $l\ge1$ because $h\le g_p\le g_l$ by (F3)), and it is tight with $S_p+h-1$ because the offsets add up to $g_p-1<\operatorname{tri}(p,l)=g_p$. These columns are distinct for distinct pairs $(l,h)$, and the excluded pairs, with $l=0$ and $h>g_0$, number at most $c=\max(0,T-g_0)\le T-1$. Hence
\[
\sum_{h=1}^{T}\bigl(|N(X_h)|-|X_h|\bigr)=\sum_{h=1}^{T}|A(K_h)|-|Y|\ \le\ |N(Y)|+c-|Y|\ \le\ T-1,
\]
so some $h$ has $|N(X_h)|\le|X_h|$, and $X_h$ is a nonempty set of starts among $S_1,\dots,S_r$.
\end{proof}

\Cref{thm:H} produces an unused tight edge rather than an unused anti-diagonal, so it does not settle \cref{prob:A}. For that we use the largest balanced set of starts. Let $\Gamma$ be the bipartite graph joining $q\in\{0,\dots,r-1\}$ to the elements of $A(q)\subseteq\{0,\dots,r\}$. It satisfies Hall's condition if $|A(K)|\ge|K|$ for all $K$, and then $K$ is \emph{balanced} if $|A(K)|=|K|$. Under Hall's condition balanced sets are closed under union: $|A(K_1\cup K_2)|+|A(K_1\cap K_2)|\le|A(K_1)|+|A(K_2)|$, and Hall's condition bounds both terms on the left from below by $|K_1\cup K_2|$ and $|K_1\cap K_2|$. Let $K^*$ be the largest balanced set, $A^*=A(K^*)$, and $B^*=\{d\ge1:g_d=1,\ d-1\in K^*\}$, the indices $d$ for which $E_d$ is one of the starts $X^*=\{S_{q+1}:q\in K^*\}$ (by (F2)). Call a level $s$ \emph{rigid} if a single block pair $\{c,d\}$ attains $M(s)$. By \cref{lem:last}, the corner $(E_c,E_d)$ and its transpose are then the only tight pairs on the anti-diagonal $M(s)$.

\smallskip
\noindent\emph{Statement $(\mathrm{E}^\flat)$ for a tail $(g_1,\dots,g_r)$: if $\Gamma$ satisfies Hall's condition and $K^*\ne\emptyset$, some rigid level $s$ has its attaining pair $\{c,d\}$ with $c\in A^*$ and $d\notin B^*$.}
\smallskip

Whether $(a,b)$ attains is decided by the signs of the differences $E_a+E_b-E_c-E_d$ with $a+b=c+d$, in which $g_0$ cancels, so $(\mathrm{E}^\flat)$ depends on $(g_1,\dots,g_r)$ only.

\begin{conjecture}[Rigid-corner certificate]\label{conj:E}
If the tight graph of a staircase has a tight permutation but not total support, then for some rigid level $s$ the corner edge $(E_c,E_d)$ lies in no tight permutation.
\end{conjecture}

\begin{theorem}\label{thm:Eflat}
Let $u(g)$, $1\le r\le k-2$, have a tight permutation but not total support, and suppose $(\mathrm{E}^\flat)$ holds for $(g_1,\dots,g_r)$. Then some rigid level $s$ satisfies the conclusion of \cref{conj:E}, the anti-diagonal $M(s)$ carries no allowed edge, and $I(u)$ is not a facet.
\end{theorem}

\begin{proof}
A tight permutation matches every start $S_{q+1}$ into $N(S_{q+1})$, so $\Gamma$ satisfies Hall's condition by \cref{lem:starts}. By \cref{thm:H} some nonempty set of starts has $|N(X)|\le|X|$, hence $|N(X)|=|X|$, and $K^*\ne\emptyset$. Take $\{c,d\}$ from $(\mathrm{E}^\flat)$. Then $E_c\in N(X^*)$, $E_d\notin X^*$ and $|N(X^*)|=|X^*|$. By (F1) every tight permutation $\tau$ maps $E_c$ into $X^*$, so $\tau(E_c)\ne E_d$; applied to $\tau^{-1}$ this gives $\tau(E_d)\ne E_c$. So neither $(E_c,E_d)$ nor $(E_d,E_c)$ is used, and these are the only tight pairs on the anti-diagonal $M(s)$. \Cref{lem:C1} finishes the proof.
\end{proof}

\begin{theorem}[computer-assisted]\label{thm:Eflat6}
Statement $(\mathrm{E}^\flat)$ holds for every tail $(g_1,\dots,g_r)$ with $r\le7$.
\end{theorem}

The proof partitions the tails with $r$ parts into finitely many polyhedral cells on which the verdict of $(\mathrm{E}^\flat)$ is constant, and checks the verdict on each nonempty cell. Fix $r\ge1$ and regard the tail $x=(g_1,\dots,g_r)$ as a vector of variables. Since $E_a-E_0=x_1+\dots+x_a$, every difference
\[
L_{ab,cd}(x)=E_a+E_b-E_c-E_d,\qquad a+b=c+d,\quad 0\le a,b,c,d\le r,
\]
is a linear form in $x$ with integer coefficients, independent of $g_0$. If $\{a,b\}\ne\{c,d\}$, say $a<c\le d<b$, then $L_{ab,cd}=(x_{d+1}+\dots+x_b)-(x_{a+1}+\dots+x_c)$ with disjoint index ranges, so the form is not zero. Let $L_1,\dots,L_N$ be these forms up to sign; $N=0,1,3,7,13,22,34$ for $r=1,\dots,7$.

\begin{lemma}\label{lem:cellverdict}
The truth of $(\mathrm{E}^\flat)$ for a tail $x$ is determined by the set $\Omega(x)=\{i:x_i=1\}$ and the sign vector $\theta(x)=(\operatorname{sign}L_j(x))_{j\le N}$.
\end{lemma}

\begin{proof}
A pair $(a,b)$ attains $M(a+b)$ iff $L_{ab,cd}(x)\ge0$ for every $(c,d)$ with $c+d=a+b$, which is read off from $\theta(x)$. So $\theta(x)$ determines the sets $A(q)$, the graph $\Gamma$, Hall's condition, the balanced sets, $K^*$, $A^*$ and the rigid levels with their attaining pairs. The set $B^*=\{d\ge1:x_d=1,\ d-1\in K^*\}$ is determined by $K^*$ and $\Omega(x)$.
\end{proof}

For $\Omega\subseteq\{1,\dots,r\}$ and $\theta\in\{-,0,+\}^N$, the \emph{cell} $C(\Omega,\theta)$ is the set of $x\in\R^r$ with $x_i=1$ for $i\in \Omega$, $x_i\ge2$ for $i\notin \Omega$, and $\operatorname{sign}L_j(x)=\theta_j$ for all $j$. Every integer tail $x$ lies in $C(\Omega(x),\theta(x))$, so this cell is nonempty. The verdict of $(\mathrm{E}^\flat)$ is defined for every pair $(\Omega,\theta)$ by the recipe of \cref{lem:cellverdict}, and by that lemma it agrees with the verdict of every integer tail in the cell. Hence $(\mathrm{E}^\flat)$ holds for all tails with $r$ parts if its verdict holds on every nonempty cell. Every nonempty cell contains an integer tail, and the cells come in pairs:

\begin{lemma}\label{lem:cellpair}
Call $\theta$ \emph{realizable} if it is the sign vector of some $x\in\R^r$. For realizable $\theta$ let $J(\theta)$ be the set of indices at which such an $x$ is minimal; it depends only on $\theta$. Then $C(\Omega,\theta)$ is nonempty if and only if $\theta$ is realizable and $\Omega=\emptyset$ or $\Omega=J(\theta)$, and both of these cells contain integer points.
\end{lemma}

\begin{proof}
The coefficients of $L_{ab,cd}$ sum to $a+b-c-d=0$, so $L_j(x+\tau\one)=L_j(x)$. For $i<j$ the pairs $(i,j-1)$ and $(i-1,j)$ have the same sum and $L_{i(j-1),(i-1)j}=x_i-x_j$, so $\theta$ fixes the order of the coordinates and hence $J(\theta)$. If $\Omega\ne\emptyset$ and $x\in C(\Omega,\theta)$, the coordinates in $\Omega$ equal $1$ and the others are at least $2$, so $\Omega=J(\theta)$. Conversely the realizations of $\theta$ form a relatively open rational cone, which contains an integer point $w$; then $2+w-\min_iw_i$ and $1+w-\min_iw_i$ (coordinatewise) lie in $C(\emptyset,\theta)$ and $C(J(\theta),\theta)$.
\end{proof}

The two cells of a sign vector can have different verdicts, because $B^*$ depends on $\Omega$; both are checked.

\begin{proof}[Proof of \cref{thm:Eflat6} (outline)]
By \cref{lem:cellverdict,lem:cellpair}, $(\mathrm{E}^\flat)$ holds for all tails with $r$ parts if and only if its verdict holds on both cells of every realizable sign vector; by Gordan's alternative, realizability is a homogeneous feasibility question. A computer enumerates the sign vectors depth first, with an integer witness for each realizable one and an infeasibility certificate for each discarded branch, and a separate exact checker verifies the witnesses, the certificates, a tiling of all branches and the verdict on every nonempty cell. The checker trusts its own implementation of the verdict recipe of \cref{lem:cellverdict}, which agrees with a second implementation by the authors on every cell for $r\le7$. The full proof is in \cref{sec:SM-Eflat6}.
\end{proof}

\begin{corollary}[computer-assisted]\label{cor:probA}
\Cref{prob:A} holds for every staircase with $r\le7$ or $r\ge k-8$, and \cref{conj:E} holds for every staircase with $r\le7$. For all these staircases, a tight permutation without total support leaves some anti-diagonal without allowed edges.
\end{corollary}

\begin{proof}
For $r\le7$ combine \cref{thm:Eflat,thm:Eflat6}; a staircase without a tight permutation has an empty face. The involution $g\mapsto g^*$ preserves facet status and total support and sends $r$ steps to $k-1-r$ steps. For the last statement with $r\ge k-8$: since $u_{J^c}=\iota-u_J$ and $u_{k-J}=|J|-u_J\circ\rho$, with $\rho(i)=k-1-i$, we have $u(g^*)=u(g)\circ\rho+\iota+c$ for a constant $c$, so the tight permutations and the allowed graph of $u(g^*)$ are the images of those of $u(g)$ under $(i,j)\mapsto(\rho(i),\rho(j))$, which maps the anti-diagonal $t$ to $2k-2-t$.
\end{proof}

Several classes of tails satisfy $(\mathrm{E}^\flat)$ for every length $r$. Write $H_j=E_j-E_0=g_1+\dots+g_j$.

\begin{proposition}\label{prop:Eflatclasses}
Statement $(\mathrm{E}^\flat)$ holds, and therefore a staircase defines a facet only if its tight graph has total support (for every $g_0$ and every $r$), in each of the following cases.
\begin{enumerate}[label=(\alph*)]
\item The tail is unimodal: $g_1\le\dots\le g_\nu\ge\dots\ge g_r$ for some $\nu$.
\item Every level $s$ has a unique maximizing pair. This holds if $E_a+E_b\ne E_c+E_d$ whenever $a+b=c+d$ and $\{a,b\}\ne\{c,d\}$, in particular if $(g_1,\dots,g_r)$ is any ordering of positive integers $\beta_1<\dots<\beta_r$ with $\beta_j>\beta_1+\dots+\beta_{j-1}$ for all $j$.
\item The prefix is superadditive: $H_p+H_q\le H_{p+q}$ whenever $p+q\le r-1$, with $g_r$ arbitrary. This holds if $g_1\le\dots\le g_{r-1}$; if $(g_1,\dots,g_{r-1})=(a,b,a,b,\dots)$ with $a\le b$; and if $g_j=a+\lfloor\alpha j\rfloor-\lfloor\alpha(j-1)\rfloor$ for $j\le r-1$, with an integer $a\ge1$ and $0\le\alpha\le1$.
\end{enumerate}
In cases (a) and (c), if $u(g)$ has a tight permutation but not total support, the anti-diagonal $2E_0$ carries no allowed edge.
\end{proposition}

The proof, an analysis of the graph $\Gamma$ in each case, is in \cref{sec:s_census-proofs}.

For nonincreasing tails, case (a) is contained in \cref{thm:dec}, which gives total support and a facet; nondecreasing tails, such as the tail $(1,2,3)$ of $g=(1,1,2,3)$, need not give facets. A failure of total support always produces a bipartite component of the allowed graph, without any facet hypothesis.

\begin{proposition}\label{prop:stairbip}
Let $u(g)$, $1\le r\le k-2$, have a tight permutation but not total support. Every nonempty set $X$ of block starts with $|N(X)|\le|X|$ (one exists by \cref{thm:H}) has $|N(X)|=|X|$ and $X\ne N(X)$, and $\ind{i\in X}-\ind{i\in N(X)}$ is a nonzero element of $Y^*(u)$. In particular $\operatorname{bip}(u)\ge1$.
\end{proposition}

The proof is in \cref{sec:s_census-proofs}.

\Cref{prop:stairbip} strengthens the first conclusion of \cref{cor:facetbip} for staircases; it says nothing about the second, $\dim Z^*(u)\ge4$. Without an empty anti-diagonal, \cref{prob:A} would need $\dim Z^*(u)\ne3+\operatorname{bip}(u)$, and \cref{prop:stairbip} does not give this.

An exhaustive exact computation shows $(\mathrm{E}^\flat)$ for all $33{,}554{,}431$ tails with $g_1+\dots+g_r\le25$, of which $8{,}177{,}376$ satisfy Hall's condition with $K^*\ne\emptyset$. So \cref{prob:A} and \cref{conj:E} hold for every staircase with $k-g_0\le25$, in particular for all $k\le26$ (computer-assisted). The largest balanced set cannot be replaced by a minimal one (\cref{sec:s_census-dict}).

For \cref{prob:B} the relevant objects are the classes of boundary steps of \cref{lem:square}(ii). Call $q\in\{0,\dots,r-1\}$ \emph{solitary} if $|A(q)|\le1$.

\begin{lemma}\label{lem:R}
If some $q\in\{0,\dots,r-1\}$ is solitary, the tight graph of $u(g)$ does not have total support.
\end{lemma}

\begin{proof}
Apply \cref{thm:H} to $X=\{S_{q+1}\}$, which has $|N(X)|=|A(q)|\le1$.
\end{proof}

\smallskip
\noindent\emph{Statement $(\mathrm{S}1)$ for a composition $g$: if the boundary steps of $u(g)$ form at least two classes, some $q\in\{0,\dots,r-1\}$ is solitary.}
\smallskip

\begin{theorem}\label{thm:S1}
If $(\mathrm{S}1)$ holds for every composition of $k$, then \cref{prob:B} holds for every staircase with $k$ indices.
\end{theorem}

\begin{proof}
Let $u(g)$ have total support. By \cref{lem:R} no index is solitary, so by $(\mathrm{S}1)$ the boundary steps of $u(g)$ form one class. The involution $g\mapsto g^*$ preserves total support, so the same holds for $u(g^*)$, and \cref{lem:square}(iv) shows that $u(g)$ is a ray.
\end{proof}

\begin{corollary}[computer-assisted]\label{cor:probB}
\Cref{prob:B} holds for all $k\le30$.
\end{corollary}

\begin{proof}
Statement $(\mathrm{S}1)$ was checked exactly for all $1{,}073{,}741{,}762$ compositions with $4\le k\le30$ and $1\le r\le k-2$, of which $483{,}464{,}468$ have at least two classes; for $k\le3$ there is a single boundary step. \Cref{thm:S1} applies.
\end{proof}

The classes of \cref{lem:square}(ii) and the sets $A(q)$ are defined by the attaining pairs, so they make sense for any real sequence $E_0,\dots,E_r$: the pair $(p,q)$ attains the level $s=p+q$ if $E_p+E_q=M(s)=\max_{a+b=s}(E_a+E_b)$, and the classes are the connected components on $\{0,\dots,r-1\}$ of the relation joining two indices that both occur in attaining pairs of one level.

\begin{theorem}[computer-assisted]\label{thm:S1cert}
Let $1\le r\le10$ and let $E_0,\dots,E_r$ be real numbers with at least two classes. Then some $q\le r-1$ has $|A(q)|\le1$. In particular $(\mathrm{S}1)$ holds for every composition with $1\le r\le10$ steps, for every $k$.
\end{theorem}

\begin{proof}[Proof of \cref{thm:S1cert} (outline)]
By hand, we take $E_0=E_1=0$ and encode the attaining pairs level by level as a mixed-radix choice; the sequences realizing a prefix of choices form a relatively open rational cone. A prefix is terminal when one of four certificates holds: the classes are already connected, some $q$ already has at most one possible partner, or the cone is empty by a zero row or by Gordan's alternative. We prove that if the terminal prefixes tile the whole range of choices, the theorem holds for that $r$. By computer, an overflow-checked generator produces the terminal prefixes and a separate exact checker, sharing no code with it, verifies the tiling and every certificate ($12{,}898{,}911$ records for $r\le9$ and $123{,}173{,}211$ for $r=10$; for $r=10$ the tiling is checked domain by domain, with a record whose prefix is shorter than the domain prefix clipped to the domain, and the domains are checked to tile the whole range). The computation does not check that the checker's encoding of the reduction (the pairs at each level, the numbering, the rows and the tests) matches the definitions; this is tested end to end on random sequences, not proved. The full proof is in \cref{sec:SM-S1cert}.
\end{proof}

The scope of \cref{thm:S1cert} for \cref{prob:B} is limited: \cref{thm:S1} uses $(\mathrm{S}1)$ for $g$ and for $g^*$, which has $k-1-r$ steps, so it gives \cref{prob:B} only for staircases with $r\le10$ and $k-1-r\le10$, which forces $k\le21$, already covered by \cref{cor:probB}.

The sufficient condition of \cref{lem:square}(iv) is not necessary: some staircases with two classes are rays (\cref{sec:s_census-proofs}).

\emph{Status of \cref{conj:T}.} The direction ``facet implies total support'' (\cref{prob:A}) is proved for $r\le7$, $r\ge k-8$ and $k\le26$, and for the tails of \cref{prop:Eflatclasses} at every depth; the direction ``total support implies facet'' (\cref{prob:B} with \cref{thm:R}) for $k\le30$. So \cref{conj:T} holds for all $k\le26$ and for $r\le3$ and $r\ge k-4$. What remains are the statements $(\mathrm{E}^\flat)$ for $r\ge8$ and $(\mathrm{S}1)$ for the compositions with $k\ge31$ and $r\ge11$ steps (\cref{thm:S1cert} covers $r\le10$), both about the integer sequence $E_0<\dots<E_r$ alone and checkable in polynomial time; \cref{prob:B} is open for $k\ge31$. \Cref{tab:evidence} summarizes the exact checks behind these statements, and the facets produced by no step potential are a separate problem (\cref{prob:C}).

%% file: b7_conclusion.tex
\section{Open problems}\label{sec:conclusion}

The facet criterion (\cref{thm:F}) decides each inequality $I(u)$ exactly, and the fiber theorem (\cref{thm:fiber}) says when distinct rays define the same facet. What is missing is a description of the facets other than the coordinate facets that does not pass through the criterion one ray at a time.

\begin{enumerate}[leftmargin=*]
\item \emph{Step potentials.} Prove \cref{conj:T}: a staircase gives a facet iff its tight graph has total support. By \cref{thm:Eflat,thm:S1} it suffices to prove $(\mathrm{E}^\flat)$ for tails with $r\ge8$ parts and $(\mathrm{S}1)$ for the compositions with $k\ge31$ and $r\ge11$ steps (\cref{thm:S1cert} covers $r\le10$ for every $k$). Both are statements about the integer sequence $E_0<\dots<E_r$ alone and can be checked for a given sequence in polynomial time. The rigid-corner statement \cref{conj:E} would give a uniform certificate for the direction of \cref{prob:A}.
\item \emph{Rays without total support.} Find the smallest $k$ for which $\Ph(k)$ has a facet ray without total support. It lies between $13$ and $19$: every facet ray has total support for $k\le12$, and \cref{ex:noTS} is one at $k=19$ (\cref{rem:samefacet}). More generally, describe the facets defined by rays whose allowed graph has a bipartite component. By \cref{thm:fiber} they are exactly the facets defined by more than one ray, and the rays that define such a facet are the vertices of a polytope whose dimension is the number of bipartite components; it is a cube in \cref{rem:cubical}.
\item \emph{Cores and padding.} Classify the cores (\cref{prob:C}(a)), which account for 2,566 of the 5,748 facets of $\Ph(12)$ that come from no step potential, and show that every normalized ray with total support has a connected tight graph and satisfies $(\mathrm{P}_L)$ and $(\mathrm{P}_R)$ (\cref{prob:C}(b)). These are the hypotheses of \cref{cor:pad,thm:refine}; they hold for every facet ray with $k\le12$, but we do not know whether total support of a ray implies them.
\item \emph{Refinement.} If $\operatorname{spr}_mu$ is a ray whose tight graph has total support, is $u$ a ray? \Cref{prop:descent} proves this under a condition on flat levels and shows that the conclusion can fail when $\operatorname{spr}_mu$ lacks total support.
\item \emph{Rays of the fan.} Li, Liu and Yao ask for a simple description of the rays of the fan \cite{LiLiuYao2026}. The facet rays are the part that matters for $\Ph(k)$; their heights grow at least quadratically (\cref{thm:height}), so no description by step potentials with bounded multiplicities can be complete. The complete ray lists for $k\le12$ are part of the deposited data (\cref{app:computation}).
\end{enumerate}

%% file: bA_computation.tex
\section{Certified computations}\label[appendix]{app:computation}

We distinguish four kinds of guarantee: \emph{exact arithmetic}; a \emph{certificate} for each outcome of the facet test; \emph{completeness} of each list of rays, a theorem about the algorithm (\cref{prop:complete}); and an \emph{independent recomputation} of each list (\cref{tab:vw}). Floating point enters only in checks whose outputs are verified exactly or which are not needed for any result.

\emph{Facet test for a potential.} Each outcome of the facet test is certified as described in \cref{sec:census}; the details, including the one-sided use of a prime modulus, are in \cref{sec:s_census-dict}.

\emph{Enumeration of rays: facet walk.} Let $k\ge4$ (for $k\le3$, $Q$ is a point or a segment and the facets of $\Ph(k)$ are found directly, \cref{sec:census}). The rays are the facet normals of $Q=\sum_tQ_t$ (\cref{sec:census}), a polytope of dimension $k-2$; we fix normals by $u(0)=u(k-1)=0$, which removes the lineality space $\spn\{\one,\iota\}$ of the normal fan. A potential $u$ is a ray if and only if the face $Q^u$ on which $u$ is minimal has dimension $k-3$ (\cref{rem:fan} and the start of \cref{sec:census}). For a Minkowski sum $N=\sum_tN_t$ of simplices, of dimension $d_N$, let $\mathcal{M}$ be the vector matroid of all differences of two vertices of one summand; for a direction $y$ let $N^y$ be the face of $N$ on which $y$ is minimal. The algorithm is:
\begin{enumerate}[leftmargin=*,label=(W\arabic*)]
\item Start from a random integer functional $w$. While the face $Q^w$ has dimension less than $k-3$, move $w$ along a random integer direction orthogonal to the affine hull of $Q^w$, by an exact line search, to the first point at which the minimal face grows. Stop when $\dim Q^w=k-3$, which gives a first facet $u$; after a fixed number of unsuccessful attempts, stop with failure.
\item At a facet $u$, form $N=Q^u=\sum_tQ_t^u$, of dimension $d_N=k-3$, and list its facets: split $\mathcal{M}$ into connected components; for each component list its flats of rank one less than its own rank, by a depth-first search that reaches each flat from its lexicographically greedy basis only; add all elements of the other components; take the two normals $\pm y$ of the resulting hyperplane of the linear span $\operatorname{lin}\mathcal{M}$ of all elements of $\mathcal{M}$, and keep those with $\dim N^y=d_N-1$.
\item Cross each listed ridge $y$ by the exact line search along $u+\lambda y$ (or along $y-\mu u$ if that ray stays in the normal cone of the ridge), which returns the facet $u'$ on the other side. If $y$ itself spans the ray of $u'$, the search returns $y$.
\item Repeat from every new facet, on orbits of the mirror $i\mapsto k-1-i$, until no unprocessed facet remains. In the recursive variant, a face $Q^u$ with at least $60$ distinct difference directions is handled in (W2) by walking its own facet graph with (W1)--(W3) instead of listing flats.
\end{enumerate}

\begin{proposition}[Completeness]\label{prop:complete}
Let $k\ge4$, and suppose the walk is started at a facet $u_0$ of $Q$ whose face dimension $\dim Q^{u_0}=k-3$ has been verified in exact arithmetic. Then, in exact arithmetic, (W2)--(W4) visit every facet of $Q$, so they list every ray of the fan. The same holds for the recursive variant, provided each of its inner walks is started at a verified facet of the face it walks.
\end{proposition}

\begin{proof}
\emph{Facets of a sum of simplices.} For every direction $y$, $N^y=\sum_tN_t^y$ \cite[Proposition~2.1]{Fukuda2004}, \cite[Lemma~4]{FukudaWeibel2007}, and $\dim N^y$ is the rank of the differences inside the $y$-minimal vertex classes. If $N^y$ is a facet, these differences have rank $d_N-1$ and are orthogonal to $y$, so the elements of $\mathcal{M}$ orthogonal to $y$ form a flat of rank at least $d_N-1$, and at most $d_N-1$ because $y$ is not orthogonal to all of $\operatorname{lin}\mathcal{M}$. So every facet is found from its flat and one of the two normals of the flat's span inside $\operatorname{lin}\mathcal{M}$, and the test $\dim N^y=d_N-1$ keeps exactly the facets. The rank-$(d_N-1)$ flats of a direct sum of connected matroids are the unions of a corank-one flat of one component with all the other components, and the depth-first search lists every such flat of a component exactly once, because every flat has a unique lexicographically greedy basis, every prefix of a greedy basis is greedy for its own closure, and a branch is pruned as soon as it is not greedy.
\emph{Crossing.} A ridge $R$ of $Q$ lies in exactly two facets, and its normal cone is two-dimensional modulo the lineality space, spanned by their normals $u$ and $u'$. For small $\lambda>0$ the point $u+\lambda y$ lies in the relative interior of this cone; the exact line search returns the first point at which the minimal face grows beyond $R$, which is on the ray of $u'$, or, if the line never leaves the cone, the direction $y$ itself lies in the cone and the search from $y$ toward $-u$ returns $u'$. In the boundary case that $y$ spans the ray of $u'$, the face of $Q$ on which $y$ is minimal is the facet of $u'$, not the ridge; the line $u+\lambda y$ never leaves the cone, the search from $y$ stops at once, and it returns $y$, which spans the ray of $u'$.
\emph{Connectivity.} The facet graph of a polytope is the vertex graph of its polar, which is connected by Balinski's theorem \cite{Balinski1961}; see also \cite[Chapter~3]{Ziegler1995}. By the first two parts, when the walk stops, every neighbor in the facet graph of every visited facet has been visited, so the visited facets form a union of connected components containing $u_0$, which is all of them. The order in which facets are processed plays no role.
\end{proof}

\emph{The first facet.} Step (W1) is randomized, and its only failure mode is to stop without a facet; every facet it returns has passed the exact test $\dim Q^w=k-3$, so it is a certified starting point for \cref{prop:complete}. In all our runs, including the inner walks of the recursive variant (\cref{sec:SM-recursive}), which also start by (W1), it returned a facet. A deterministic start is also available: by \cref{prop:hinges}(i) the hinge $h_1(i)=(i-1)^+$ is a ray for every $k\ge3$, so $Q^{h_1}$ is a facet of $Q$ (\cref{rem:fan} and the duality between rays of the normal fan and facets of $Q$ in \cref{sec:census}); we checked $\dim Q^{h_1}=k-3$ exactly for $4\le k\le15$.

\emph{Exact arithmetic.} Normals and line searches use exact integers; in the walk, ranks and face dimensions are computed modulo the prime $2^{30}-35$ (the facet test uses $2^{31}-1$, \cref{sec:s_census-dict}), which by Hadamard's inequality gives the rational rank for $k\le25$ (\cref{sec:s_census-cross}). Every enumeration also checks that each new face $Q^u$ has dimension $k-3$ and that every ridge is found from both sides; the second check cannot detect a ridge missed from both sides, which is why completeness rests on \cref{prop:complete}.

\emph{Independent recomputation: vertex walk.} The rays are also the extreme rays of the normal cones of the vertices of $Q$, and we walked the vertex graph of $Q$ in two versions that share no code with (W1)--(W4): one computes each cone with cddlib in floating point and verifies every reported ray exactly, the other uses exact integer arithmetic throughout (\cref{sec:SM-vertexwalk}). Both agree with the facet walk for $5\le k\le\VWEXACTK$ (\cref{tab:vw}). So for $5\le k\le12$ every ray list behind \cref{tab:counts,tab:census} is confirmed by a second, independent exact enumeration. Which rays give facets is decided by the certified facet test described above, rerun on every ray in the reproduction below; for $k\le8$ the facet lists also agree with the floating-point vertex--facet enumerations of \cref{tab:counts}.

\begin{table}[ht]
\centering
\caption{Independent recomputation of the ray lists by a walk on the vertices of $Q$ (cddlib version, rays verified exactly). The last column reports whether the exact-integer version reproduces the same ray set.}
\label{tab:vw}
\begin{tabular}{@{}rrrrrr@{}}
\toprule
$k$ & vertices of $Q$ & rays (vertex walk) & rays (facet walk) & identical & exact version\\
\midrule
5 & 22 & 16 & 16 & yes & yes\\
6 & 100 & 42 & 42 & yes & yes\\
7 & 484 & 110 & 110 & yes & yes\\
8 & 2,689 & 338 & 338 & yes & yes\\
9 & 15,732 & 1,146 & 1,146 & yes & yes\\
10 & 97,690 & 4,370 & 4,370 & yes & yes\\
11 & 632,678 & 18,515 & 18,515 & yes & yes\\
12 & 4,363,898 & 89,321 & 89,321 & yes & yes\\
\bottomrule
\end{tabular}
\end{table}

\emph{Further checks.} Further cross-checks and the costs of all runs are in \cref{sec:s_census-cross}, the cell enumeration in \cref{app:cells}, and the certificate formats in \cref{sec:s_census-certs}.

\emph{Availability.} The code, the complete ray and facet lists for $k\le12$ with their certificates, the scripts for every check in \cref{tab:evidence}, \cref{tab:vw} and \cref{tab:cores}, the cell enumeration of \cref{app:cells} with its checks, the certificate streams for \cref{thm:Eflat6} and for $(\mathrm{S}1)$ with $r\le9$, and the hashes, generator and checker for $r=10$, with the $r=10$ streams as a separate file, are deposited in \ifanon an anonymized repository for review\else a permanent archive, \url{https://doi.org/10.5281/zenodo.23440726}\fi.

\emph{Reproduction.} Every computation behind the paper was rerun from the deposited code in one pass of independent pieces, on two machines: Linux x86\_64 cluster nodes (four lanes of 16 cores; 80 CPU-hours, about 6 hours elapsed) and an Apple Mac mini with an M6 processor (12 cores, 16~GB of memory; 47 CPU-hours over about 8 hours of running time, with a largest process of 3.7~GB). Every regenerated file, compared by the SHA-256 of its uncompressed content, and every computed number agree between the two runs and with the deposited data; the archive gives the cost of each piece.

%% file: main.bbl
\providecommand{\bysame}{\leavevmode\hbox to3em{\hrulefill}\thinspace}
\providecommand{\MR}{\relax\ifhmode\unskip\space\fi MR }
\providecommand{\MRhref}[2]{%
  \href{http://www.ams.org/mathscinet-getitem?mr=#1}{#2}
}
\providecommand{\href}[2]{#2}
\begin{thebibliography}{10}

\bibitem{AraozEvansGomoryJohnson2003}
Julian Ar{\'a}oz, Lisa Evans, Ralph~E. Gomory, and Ellis~L. Johnson,
  \emph{Cyclic group and knapsack facets}, Math. Program. \textbf{96} (2003),
  no.~2, 377--408.

\bibitem{ArdilaDevelin2009}
Federico Ardila and Mike Develin, \emph{Tropical hyperplane arrangements and
  oriented matroids}, Math. Z. \textbf{262} (2009), no.~4, 795--816.

\bibitem{Balinski1961}
M.~L. Balinski, \emph{On the graph structure of convex polyhedra in $n$-space},
  Pacific J. Math. \textbf{11} (1961), no.~2, 431--434.

\bibitem{BasuHildebrandKoppe2016a}
Amitabh Basu, Robert Hildebrand, and Matthias K{\"o}ppe, \emph{Light on the
  infinite group relaxation {I}: foundations and taxonomy}, 4OR \textbf{14}
  (2016), no.~1, 1--40.

\bibitem{BasuHildebrandKoppe2016b}
Amitabh Basu, Robert Hildebrand, and Matthias K{\"o}ppe, \emph{Light on the
  infinite group relaxation {II}: sufficient conditions for extremality,
  sequences, and algorithms}, 4OR \textbf{14} (2016), no.~2, 107--131.

\bibitem{BebeacuaMansourPostnikovSeverini2005}
Cecilia Bebeacua, Toufik Mansour, Alexander Postnikov, and Simone Severini,
  \emph{On the {X}-rays of permutations}, Electron. Notes Discrete Math.
  \textbf{20} (2005), 193--203, arXiv:math/0506334.

\bibitem{BrualdiFritscher2014}
Richard~A. Brualdi and Eliseu Fritscher, \emph{Hankel and {T}oeplitz {X}-rays
  of permutations}, Linear Algebra Appl. \textbf{449} (2014), 350--380.

\bibitem{BrualdiGibson1977}
Richard~A. Brualdi and Peter~M. Gibson, \emph{Convex polyhedra of doubly
  stochastic matrices. {I}. {A}pplications of the permanent function}, J.
  Combin. Theory Ser. A \textbf{22} (1977), no.~2, 194--230.

\bibitem{BrunettiDelLungoGritzmannDeVries2008}
Sara Brunetti, Alberto Del~Lungo, Peter Gritzmann, and Sven de~Vries, \emph{On
  the reconstruction of binary and permutation matrices under (binary)
  tomographic constraints}, Theoret. Comput. Sci. \textbf{406} (2008),
  no.~1--2, 63--71.

\bibitem{CKcompanion}
Prashant Chaudhary and Kapil Khandelwal, \emph{Sums of two dependent uniform
  random variables: a characterization by joint tail bounds}, arXiv:2610.24197,
  2026.

\bibitem{DevelinSturmfels2004}
Mike Develin and Bernd Sturmfels, \emph{Tropical convexity}, Doc. Math.
  \textbf{9} (2004), 1--27, Erratum: Doc. Math. 9 (2004) 205--206.

\bibitem{DulmageMendelsohn1958}
A.~L. Dulmage and N.~S. Mendelsohn, \emph{Coverings of bipartite graphs},
  Canad. J. Math. \textbf{10} (1958), 517--534.

\bibitem{FinkRincon2015}
Alex Fink and Felipe Rinc{\'o}n, \emph{Stiefel tropical linear spaces}, J.
  Combin. Theory Ser. A \textbf{135} (2015), 291--331.

\bibitem{Fukuda2004}
Komei Fukuda, \emph{From the zonotope construction to the {M}inkowski addition
  of convex polytopes}, J. Symbolic Comput. \textbf{38} (2004), no.~4,
  1261--1272.

\bibitem{FukudaWeibel2007}
Komei Fukuda and Christophe Weibel, \emph{f-vectors of {M}inkowski additions of
  convex polytopes}, Discrete Comput. Geom. \textbf{37} (2007), no.~4,
  503--516.

\bibitem{Gomory1969}
Ralph~E. Gomory, \emph{Some polyhedra related to combinatorial problems},
  Linear Algebra Appl. \textbf{2} (1969), no.~4, 451--558.

\bibitem{GritzmannSturmfels1993}
Peter Gritzmann and Bernd Sturmfels, \emph{{M}inkowski addition of polytopes:
  computational complexity and applications to {G}r{\"o}bner bases}, SIAM J.
  Discrete Math. \textbf{6} (1993), no.~2, 246--269.

\bibitem{HardyRamanujan1918}
G.~H. Hardy and S.~Ramanujan, \emph{Asymptotic formulae in combinatory
  analysis}, Proc. London Math. Soc. (2) \textbf{17} (1918), no.~1, 75--115.

\bibitem{LiLiuYao2026}
Bangzheng Li, Yuewei Liu, and Yuan Yao, \emph{Binary {X}-rays of doubly
  stochastic matrices}, arXiv:2608.01442v2, 2026.

\bibitem{MaoWangWang2019}
Tiantian Mao, Bin Wang, and Ruodu Wang, \emph{Sums of standard uniform random
  variables}, J. Appl. Probab. \textbf{56} (2019), no.~3, 918--936,
  arXiv:1811.11981.

\bibitem{Nordh2017}
Gustav Nordh, \emph{A note on {X}-rays of permutations and a problem of
  {B}rualdi and {F}ritscher}, arXiv:1707.03928, 2017.

\bibitem{Paffenholz2015}
Andreas Paffenholz, \emph{Faces of {B}irkhoff polytopes}, Electron. J. Combin.
  \textbf{22} (2015), no.~1, Paper 1.67.

\bibitem{PerroneSolusUhler2019}
Elisa Perrone, Liam Solus, and Caroline Uhler, \emph{Geometry of discrete
  copulas}, J. Multivariate Anal. \textbf{172} (2019), 162--179,
  arXiv:1802.06969.

\bibitem{SinkhornKnopp1967}
Richard Sinkhorn and Paul Knopp, \emph{Concerning nonnegative matrices and
  doubly stochastic matrices}, Pacific J. Math. \textbf{21} (1967), no.~2,
  343--348.

\bibitem{Ziegler1995}
G{\"u}nter~M. Ziegler, \emph{Lectures on polytopes}, Graduate Texts in
  Mathematics, vol. 152, Springer, New York, 1995.

\end{thebibliography}
